\documentclass[11pt]{article}
\usepackage[latin9]{inputenc}
\usepackage{geometry}
\usepackage{color}
\usepackage{amsmath}
\usepackage{amsthm}
\usepackage{amssymb}
\usepackage[unicode=true,pdfusetitle,
 bookmarks=true,bookmarksnumbered=false,bookmarksopen=false,
 breaklinks=false,pdfborder={0 0 1},backref=page,colorlinks=true]
 {hyperref}
\hypersetup{
 linkcolor=blue}

\makeatletter
\numberwithin{equation}{section}
\numberwithin{figure}{section}
\theoremstyle{plain}
\newtheorem{thm}{\protect\theoremname}[section]
\theoremstyle{remark}
\newtheorem{rem}[thm]{\protect\remarkname}
\theoremstyle{plain}
\newtheorem{conjecture}[thm]{\protect\conjecturename}
\theoremstyle{plain}
\newtheorem{question}[thm]{\protect\questionname}
\theoremstyle{definition}
\newtheorem{defn}[thm]{\protect\definitionname}
\theoremstyle{plain}
\newtheorem{prop}[thm]{\protect\propositionname}
\theoremstyle{plain}
\newtheorem{lem}[thm]{\protect\lemmaname}
\theoremstyle{remark}
\newtheorem{notation}[thm]{\protect\notationname}
\newenvironment{lyxlist}[1]
	{\begin{list}{}
		{\settowidth{\labelwidth}{#1}
		 \setlength{\leftmargin}{\labelwidth}
		 \addtolength{\leftmargin}{\labelsep}
		 }}
	{\end{list}}
\newcommand{\lyxaddress}[1]{
	\par {\raggedright #1
	\vspace{1.4em}
	\noindent\par}
}

\@ifundefined{date}{}{\date{}}
\usepackage{graphicx}
\usepackage{appendix}

\usepackage{enumitem}

\usepackage[backref=page]{hyperref}

\makeatother

\providecommand{\conjecturename}{Conjecture}
\providecommand{\definitionname}{Definition}
\providecommand{\lemmaname}{Lemma}
\providecommand{\notationname}{Notation}
\providecommand{\propositionname}{Proposition}
\providecommand{\questionname}{Question}
\providecommand{\remarkname}{Remark}
\providecommand{\theoremname}{Theorem}

\begin{document}
\global\long\def\F{\mathrm{\mathbf{F}} }%
\global\long\def\Aut{\mathrm{Aut}}%
\global\long\def\C{\mathbf{C}}%
\global\long\def\H{\mathbb{H}}%
\global\long\def\U{\mathbf{U}}%
\global\long\def\P{\mathcal{P}}%

\global\long\def\trace{\mathrm{tr}}%
\global\long\def\End{\mathrm{End}}%

\global\long\def\L{\mathcal{L}}%
\global\long\def\W{\mathcal{W}}%
\global\long\def\E{\mathbb{E}}%
\global\long\def\SL{\mathrm{SL}}%
\global\long\def\R{\mathbf{R}}%
\global\long\def\Pairs{\mathrm{PowerPairs}}%
\global\long\def\Z{\mathbf{Z}}%
\global\long\def\rs{\to}%
\global\long\def\A{\mathcal{A}}%
\global\long\def\a{\mathbf{a}}%
\global\long\def\rsa{\rightsquigarrow}%
\global\long\def\ep{\varepsilon}%
\global\long\def\b{\mathbf{b}}%
\global\long\def\df{\mathrm{def}}%
\global\long\def\eqdf{\stackrel{\df}{=}}%
\global\long\def\ZZ{\mathcal{Z}}%
\global\long\def\Tr{\mathrm{Tr}}%
\global\long\def\N{\mathbf{N}}%
\global\long\def\std{\mathrm{std}}%
\global\long\def\HS{\mathrm{H.S.}}%
\global\long\def\e{\mathbf{e}}%
\global\long\def\c{\mathbf{c}}%
\global\long\def\d{\mathbf{d}}%
\global\long\def\AA{\mathbf{A}}%
\global\long\def\BB{\mathbf{B}}%
\global\long\def\u{\mathbf{u}}%
\global\long\def\v{\mathbf{v}}%
\global\long\def\spec{\mathrm{spec}}%

\title{Spectral gap for Weil-Petersson random surfaces with cusps}
\author{Will Hide }

\maketitle
\date{\vspace{-5ex}}
\begin{abstract}
We show that for any $\ep>0$, $\alpha\in[0,\frac{1}{2})$, as $g\to\infty$
a generic finite-area genus $g$ hyperbolic surface with $n=O\left(g^{\alpha}\right)$
cusps, sampled with probability arising from the Weil-Petersson metric
on moduli space, has no non-zero eigenvalue of the Laplacian below
$\frac{1}{4}-\left(\frac{2\alpha+1}{4}\right)^{2}-\ep$. For $\alpha=0$
this gives a spectral gap of size $\frac{3}{16}-\ep$ and for any
$\alpha<\frac{1}{2}$ gives a uniform spectral gap of explicit size.
\end{abstract}
{\small{}\tableofcontents{}}{\small\par}

\section{Introduction}

A hyperbolic surface is a smooth, connected, orientable Riemannian
surface with constant Gaussian curvature $-1$. Let $X$ be a finite-area
non-compact hyperbolic surface. The $L^{2}(X)$ spectrum of the Laplacian
$\Delta_{X}$ consists of: 
\begin{itemize}
\item A simple eigenvalue at $0$ and possibly finitely many eigenvalues
in $\Big(0,\frac{1}{4}\Big)$.
\item Absolutely continuous spectrum $\Big[\frac{1}{4},\infty\Big)$ with
multiplicity equal to the number of cusps of $X$.
\item Possibly infinitely many discrete eigenvalues in $\Big[\frac{1}{4},\infty\Big)$,
embedded in the absolutely continuous spectrum.
\end{itemize}
Spectral gap refers to the gap between the zero eigenvalue and the
remaining spectrum. The spectral gap is closely related to the connectivity
of a surface and the rate of mixing of the geodesic flow. We are interested
in the size of the spectral gap for a random surface with large genus.
The random model we shall consider is the Weil-Petersson model, arising
from the Weil-Petersson metric on moduli space, explained in Section
\ref{sec:Geometric-background}.

Motivation for this paper arises from recent results for compact surfaces.
In contrast to our setting, the spectrum of a compact hyperbolic surface
$Y$ consists of eigenvalues 
\[
0=\lambda_{0}\left(Y\right)<\lambda_{1}\left(Y\right)\leqslant\dots\leqslant\lambda_{k}\left(Y\right)\leqslant\dots,
\]
with $\lambda_{j}\left(Y\right)\to\infty$ as $j\to\infty$. The first
spectral gap result for Weil-Petersson random compact surfaces was
due to Mirzakhani in \cite{Mi13}, who proved the following.
\begin{thm}[Mirzakhani '13]
The Weil-Petersson probability that a genus $g$ compact hyperbolic
surface has a non-zero Laplacian eigenvalue below $\frac{1}{4}\left(\frac{\log(2)}{2\pi+\log(2)}\right)^{2}\approx0.0024$
tends to zero as $g\to\infty$.
\end{thm}

Recently this result was improved, independently by Wu and Xue in
\cite{WX21} and Lipnowski and Wright in \cite{LW21} to the following.
\begin{thm}[Wu-Xue and Lipnowski-Wright '21]
\label{Compact case}For any $\ep>0$, the Weil-Petersson probability
that a genus $g$ compact hyperbolic surface has a non-zero Laplacian
eigenvalue below $\frac{3}{16}-\ep$ tends to zero as $g\to\infty$.
\end{thm}

The purpose of this paper is to extend Theorem \ref{Compact case}
to non-compact finite-area surfaces. We prove the following.
\begin{thm}
\label{Main Theorem of paper}For any $0\leqslant\alpha<\frac{1}{2}$,
if $n=O\left(g^{\alpha}\right)$ then for any $\ep>0$ the Weil-Petersson
probability that a genus $g$ non-compact finite-area surface with
$n$ cusps has a non-zero Laplacian eigenvalue below $\frac{1}{4}-\left(\frac{2\alpha+1}{4}\right)^{2}-\ep$
tends to zero as $g\to\infty$.
\end{thm}

When $\alpha=0$, i.e. the number of cusps is bounded as $g\to\infty$,
Theorem \ref{Main Theorem of paper} returns a spectral gap of size
$\frac{3}{16}-\ep$ as in Theorem \ref{Compact case}. For any $\alpha<\frac{1}{2}$,
Theorem \ref{Main Theorem of paper} gives an explicit positive uniform
spectral gap. 

The hypothesis $n=O\left(g^{\alpha}\right)$ for $0\leqslant\alpha<\frac{1}{2}$
has geometric consequences in terms of Benjamini-Schramm convergence.
In \cite[Corollary 4.4]{Mo21b}, Monk proved that with high probability,
Weil-Petersson random surfaces with genus $g$ and $n=O\left(g^{\alpha}\right)$
cusps Benjamini-Schramm converge to the hyperbolic plane. The regime
$n=O\left(g^{\alpha}\right)$ with $0\leqslant\alpha<\frac{1}{2}$
is studied by Le Masson and Sahlsten in \cite{LS20} where they prove
a quantum ergodicity result for eigenfunctions of the Laplacian.
\begin{rem}
\label{Shen and Wu}Due to a recent work of Shen and Wu \cite{SW22},
the hypothesis $n=O\left(g^{\alpha}\right)$ for $0\leqslant\alpha<\frac{1}{2}$
cannot be relaxed much further. In particular, they prove that if
$n$ satisfies $n\gg g^{\frac{1}{2}+\beta}$ for some $\beta>0$ then
for any $\ep>0$, a Weil-Petersson random surface with genus $g$
and $n$ cusps has a non-zero eigenvalue below $\ep$ with probability
tending to $1$ as $g\to\infty$. They also prove the analogous result
for $g$ fixed and $n\to\infty$.
\end{rem}

\subsection{Other related works}

The first spectral gap result for random surfaces was due to Brooks
and Makover \cite{BM04}. They considered a random closed surface
formed by gluing $2n$ copies of an ideal hyperbolic triangle with
gluing determined by a random trivalent ribbon graph and then applying
a compactification procedure. They proved the existence of a non-explicit
constant $C>0$ such that the first non-zero eigenvalue is greater
than $C$ with probability tending to $1$ as $n\to\infty$.

\subsection*{Spectral theory in the Weil-Petersson model}

The work of Monk in \cite{Mo21a} gives estimates on the density of
Laplace eigenvalues below $\frac{1}{4}$ for Weil-Petersson random
compact surfaces. In \cite{GMST21}, Gilmore, Le Masson, Sahlsten
and Thomas obtain bounds for the $L^{p}$ norms of Laplace eigenfunctions
for Weil-Petersson random compact surfaces.

\subsection*{$\textbf{Random covers}$}

In \cite{MN20}, Magee and Naud introduced a model of a random surface
by picking a base surface $X$ and considering random degree $n$
covers $X_{n}$ of $X$, sampled uniformly. Building on work from
\cite{MP20}, in \cite{MNP20}, Magee, Naud and Puder prove that for
$X$ compact, $X_{n}$ has no new eigenvalues of the Laplacian below
$\frac{3}{16}-\epsilon$ with probability tending towards one as $n\to\infty$.
Following an intermediate result \cite{MN20}, Magee and Naud prove
in \cite{MN21} that for $X$ conformally compact, $X_{n}$ has no
new resonances in any compact set $\mathcal{K}\subset\{s\mid\text{Re}(s)>\frac{\delta}{2}\}$
with probability tending to $1$ as $n\to\infty$, where $\delta$
is the Hausdorff dimension of the limit set of $\Gamma_{X}$. In contrast
to our setting, a conformally compact hyperbolic surface has infinite
area and no cusps.

\subsection*{$\textbf{Selberg's eigenvalue conjecture}$}

Spectral theory of the Laplacian on arithmetic hyperbolic surfaces
has important consequences in Number Theory, see e.g. \cite{Sa03}.
Let $N\geqslant1$, the principal congruence subgroup of $\text{SL}_{2}(\mathbb{Z})$
of level $N$ is 
\[
\Gamma\left(N\right)=\left\{ T\in\text{SL}_{2}(\mathbb{Z})\mid T\equiv I\mod N\right\} .
\]
Consider the quotient $X\left(N\right)\stackrel{\text{def}}{=}\Gamma\left(N\right)\backslash\mathbb{H}$.
For $N>2$, the quotient $X\left(N\right)$ is a finite-area non-compact
hyperbolic surface with the number of cusps $n(N)>0$ given by
\[
n(N)=\frac{N^{2}}{2}\prod_{p\mid N}\left(1-p^{-2}\right),
\]
and genus 
\[
g(N)=1+\frac{\left(N-6\right)N^{2}}{24}\prod_{p\mid N}\left(1-p^{-2}\right),
\]
by e.g. \cite[Theorem 2.12]{Be16}. Letting $\lambda_{1}\left(X\left(N\right)\right)$
denote the first non-zero eigenvalue of the Laplacian on $X\left(N\right)$,
in \cite{Se65} Selberg made the following conjecture.
\begin{conjecture}
For every $N\geqslant1$, \label{Selberg's Conjecture}
\[
\lambda_{1}\left(X(N)\right)\geqslant\frac{1}{4}.
\]
\end{conjecture}

Conjecture \ref{Selberg's Conjecture} remains open however there
have been a number of results in this direction. Selberg proved in
\cite{Se65} that Conjecture \ref{Selberg's Conjecture} holds with
the bound $\frac{3}{16}$. After many intermediate results \cite{GJ78,Iw89,LRS95,Sa95,Iw96,KS02},
the best known result is the following due to Kim and Sarnak \cite{Ki03}.
\begin{thm}[Kim-Sarnak '03]
For every $N\geqslant1$,
\[
\lambda_{1}\left(X(N)\right)\geqslant\frac{975}{4096}.
\]
\end{thm}

In light of this, it would be interesting to know if Theorem \ref{Main Theorem of paper}
can be extended to the case that the number of cusps satisfies $n\sim g^{\frac{2}{3}}$.
\begin{question}
\label{que:Does-a-Weil-Petersson}Does a Weil-Petersson random surface
with genus $g$ and $n\sim g^{\frac{2}{3}}$ cusps have a uniform
positive spectral gap as $g\to\infty$?
\end{question}

\begin{rem}
Since the preprint version of the current paper first appeared in
July 2021, Question \ref{que:Does-a-Weil-Petersson} has been answered
in the negative by Shen and Wu \cite{SW22}, c.f. Remark \ref{Shen and Wu}.
\end{rem}

\subsection{Structure of the paper }

In the compact case, both proofs of Theorem \ref{Compact case}, in
\cite{WX21} and \cite{LW21}, rely on Selberg's trace formula, e.g.
\cite[9.5.3]{Bu92} to relate the Laplacian eigenvalues of a surface
to its length spectrum. In the non-compact finite-area setting, there
is a version of Selberg's trace formula, e.g. \cite[Theorem 10.2]{Iw02},
but it is more complicated with additional terms related to the absolutely
continuous spectrum. It is not clear to the author how to control
these additional terms. To get around this, in Section \ref{sec:Analytic-preparations}
we prove that if a surface $X\in\mathcal{M}_{g,n}$ has $\lambda_{1}\left(X\right)\leqslant\frac{3}{16}$
then $\lambda_{1}\left(X\right)$ satisfies an inequality (Theorem
\ref{Theorem: Main analytic Theorem }) involving the set of oriented
primitive closed geodesics $\mathcal{P}(X)$, which closely resembles
the form of Selberg's trace formula for compact surfaces, up to well
behaved error terms depending only the topology of the surface. Roughly
we prove that there are strictly positive functions $R$ and $f$
such that
\begin{equation}
R\left(\lambda_{1}\left(X\right),g,n\right)\leqslant\sum_{\gamma\in\mathcal{P}(X)}\sum_{k=1}^{\infty}\frac{l_{\gamma}\left(X\right)}{2\sinh\left(\frac{kl_{\gamma}(x)}{2}\right)}f\left(kl_{\gamma}\left(X\right)\right),\label{eq:Trace inequality introduction}
\end{equation}
where $l_{\gamma}\left(X\right)$ is the length of the geodesic $\gamma\in\mathcal{P}\left(X\right)$.
The proof of Theorem \ref{Theorem: Main analytic Theorem } relies
on results from \cite{Ga02}. The function $R$ is large for small
$\lambda_{1}\left(X\right)$ and bounding the Weil-Petersson expectation
of the right hand side of (\ref{eq:Trace inequality introduction})
will allow us to conclude Theorem \ref{Main Theorem of paper} through
Markov's inequality. 

After we have established Theorem \ref{Theorem: Main analytic Theorem },
we can proceed as in the compact case, making the necessary adaptations
along the way. Section \ref{sec:Geometric-background} introduces
the necessary geometric background on moduli space, the Weil-Petersson
model and Mirzakhani's integration formula. Then in Section \ref{sec:Geometric-estimates}
we bound the Weil-Petersson expectation of the right hand side of
(\ref{eq:Trace inequality introduction}), closely following the approach
taken in \cite{WX21}. Finally, in Section \ref{sec:Proof-of-Theorem}
we apply Markov's inequality to bound the probability that $X\in\mathcal{M}_{g,n}$
has a small eigenvalue to conclude the proof of Theorem \ref{Main Theorem of paper}.

In order to deduce Theorem \ref{Main Theorem of paper}, we need to
be able to estimate expressions involving the Weil-Petersson volumes
$V_{g,n}$ where $n$ grows with $g$, which is the focus of the Appendix
\ref{sec:Volume-estimates}. 

\subsection{Notation}

For real valued functions $f,h$ depending on a parameter $g$ we
write $f\ll h$ or $f=O\left(h\right)$ if there exists $C,G>0$ such
that $|f(g)|\leqslant Ch(g)$ for all $g>G$. We add subscripts to
the $\ll$ sign if the constant $C,G$ depend on another variable.
E.g. we write $f\ll_{\epsilon}h$ if exists $C=C(\ep),G=G(\ep)$ such
that $|f(g)|\leqslant Ch(g)\text{ for all }g>G$. We write $f\sim h$
if $f\ll h$ and $h\ll f$. We write $(\underline{0}_{j},a_{1},\dots,a_{k})$
to denote $(0,\dots,0,a_{1},\dots,a_{k})\in\mathbb{R}^{j+k}$ and
we write $\mathbb{R}_{\geqslant0}$ (resp. $\mathbb{Z}_{\geqslant0}$)
to denote the non-negative real numbers (resp. integers). 

\subsection*{Acknowledgments}

We thank Michael Magee and Joe Thomas for discussions about this work.
We thank the anonymous referees for their careful reading and comments
that have improved the paper.

\section{Analytic preparations\label{sec:Analytic-preparations}}

In this section we develop the necessary analytic machinery to prove
Theorem \ref{Main Theorem of paper}. We prove a version of Selberg's
trace formula, using a pre-trace inequality in place of the usual
pre-trace formula.

In Section \ref{subsec:Test-functions} we exhibit a family of test
functions $f_{T}$ where $T=4\log g$, and $f_{T}$ is a non-negative,
even, smooth function with support exactly $\left(-T,T\right)$ whose
Fourier transform $\hat{f}_{T}$ is non-negative on $\mathbb{R}\cup i\mathbb{R}$
with $\hat{f}_{T}\left(\frac{i}{2}\right)=O\left(g^{2}\right)$. The
family of test functions $f_{T}$ is defined by (\ref{Definition of g_T})
with $T=4\log g$.

The goal of this section is to prove the following. 
\begin{thm}
\label{Theorem: Main analytic Theorem }For $g\geqslant2$, let $f_{T}$
be the test function defined by (\ref{Definition of g_T}) with $T=4\log g$.
For any $\varepsilon>0$, there exists a constant $C(\varepsilon)>0$
such that for any non-compact finite-area surface $X$ with genus
$g$, $n=o\left(g^{\frac{1}{2}}\right)$ cusps and $\lambda_{1}(X)\leqslant\frac{3}{16}$,
\begin{equation}
C(\varepsilon)\log\left(g\right)g^{4(1-\varepsilon)\sqrt{\frac{1}{4}-\lambda_{1}(X)}}\leqslant\sum_{\gamma\in\mathcal{P}(X)}\sum_{k=1}^{\infty}\frac{l_{\gamma}\left(X\right)}{2\sinh\left(\frac{kl_{\gamma}(x)}{2}\right)}f_{T}\left(kl_{\gamma}\left(X\right)\right)-\hat{f}_{T}\left(\frac{i}{2}\right)+O\left(ng\right).\label{Main analytic theorem}
\end{equation}
\end{thm}

The left hand side depends on $\lambda_{1}(X)$ and the right hand
side depends on the the length spectrum of $X$. 
\begin{rem}
Given $\kappa>0$, we could have stated Theorem \ref{Theorem: Main analytic Theorem }
with the hypothesis $\lambda_{1}(X)\leqslant\frac{1}{4}-\kappa$,
(the statement is almost the same except the constant $C(\varepsilon)$
will also depend on $\kappa$) however our geometric estimates (Section
\ref{sec:Geometric-estimates}) are not strong enough to prove a spectral
gap larger than $\frac{3}{16}$. We therefore state Theorem \ref{Theorem: Main analytic Theorem }
with the hypothesis $\lambda_{1}(X)\leqslant\frac{3}{16}$ to simplify
notation.
\end{rem}

\subsection{The Laplacian on hyperbolic surfaces}

Consider the upper half plane
\[
\mathbb{H}=\{x+iy\mid x,y\in\mathbb{R},y>0\},
\]
with metric given by 
\[
\frac{dx^{2}+dy^{2}}{y^{2}}.
\]
The orientation preserving isometry group of $\mathbb{H}$ is $\text{PSL}_{2}(\mathbb{R})$,
acting via Möbius transformations. The Laplacian on $\mathbb{H}$,
denoted $\Delta_{\mathcal{\mathbb{H}}}$, is given by 
\[
\Delta_{\mathcal{\mathbb{H}}}=-y^{2}\left(\frac{\partial^{2}}{\partial x^{2}}+\frac{\partial^{2}}{\partial y^{2}}\right).
\]
 A non-compact finite-area hyperbolic surface can be realized as a
quotient $\Gamma_{X}\backslash\mathbb{H}$ where $\Gamma_{X}$ is
a finitely generated discrete free subgroup of $\text{PSL}_{2}(\mathbb{R})$,
containing parabolic elements (elements with trace $\pm2$). $\Delta_{\mathcal{\mathbb{H}}}$
is invariant under the action of $\text{PSL}_{2}(\mathbb{R})$ and
descends to an operator on $C_{c}^{\infty}(X).$ It extends uniquely
to a non-negative unbounded self-adjoint operator on $L^{2}(X).$
We let $\Delta_{X}$ denote the Laplacian on $X$ and write $\spec\left(\Delta_{X}\right)$
for the spectrum of $\Delta_{X}$. We write $\lambda_{j}\left(X\right)$
to denote the $j$th smallest non-zero eigenvalue of $\Delta_{X}$
if it exists. 

A parabolic cylinder is the quotient of $\mathbb{H}$ by a parabolic
cyclic group. We define a cusp to be the small end of a parabolic
cylinder, with boundary the unique closed horocycle of length $1$.
By \cite[Lemma 4.4.6]{Bu92}, in any finite-area hyperbolic surface,
cusps must be pairwise disjoint. Throughout Section \ref{sec:Analytic-preparations}
we let $X=\Gamma_{X}\backslash\mathbb{H}$ be a fixed non-compact
finite-area hyperbolic surface with genus $g$ and $n=o\left(g^{\frac{1}{2}}\right)$
cusps and, for the sake of argument, $\lambda_{1}(X)\leqslant\frac{3}{16}$.

\subsection{Fundamental domains}

In this subsection we introduce a decomposition of the fundamental
domain which we will need in the proof of Theorem \ref{Theorem: Main analytic Theorem }.
We shall closely follow \cite[Section 2.2]{Iw02}, and refer the reader
there for all of the notions introduced in this subsection.

We write $\mathcal{F}$ to denote a Dirichlet fundamental domain for
$\Gamma_{X}$. Since $\mathcal{F}$ is a non-compact polygon, it has
some of its vertices on $\mathbb{R}\cup\infty$ in $\mathbb{H\cup\partial}\mathbb{H}$.
We call such a vertex a cuspidal vertex. By e.g. \cite[Proposition 2.4]{Iw02},
we can ensure that the cuspidal vertices are distinct modulo $\Gamma_{X}$.
The sides of $\mathcal{F}$ can be arranged in pairs so that the side
pairing motions generate $\Gamma_{X}$. The two sides of $\mathcal{F}$
meeting at a cuspidal vertex have to be pairs since the cuspidal vertices
are distinct modulo $\Gamma_{X}$. The side-pairing motion has to
fix the vertex and is therefore a parabolic element of $\Gamma_{X}$.
This gives rise to a cusp in the quotient $\Gamma_{X}\backslash\mathbb{H}$
and each cuspidal vertex corresponds to a unique cusp in this way.
We label the cuspidal vertices by $\mathfrak{a}_{1},...,\mathfrak{a}_{n}$.
We denote the stabilizer subgroup of the vertex $\mathfrak{a}_{i}$
by
\[
\Gamma_{\mathfrak{a}_{i}}\stackrel{\text{def}}{=}\{\gamma\in\Gamma_{X}\mid\gamma\mathfrak{a}_{i}=\mathfrak{a}_{i}\}.
\]
Each $\Gamma_{\mathfrak{a}_{i}}$ is an infinite cyclic group generated
by the parabolic element $\gamma_{\mathfrak{a}_{i}}$, which is the
side-pairing motion at the vertex $\mathfrak{a}_{i}$. There exists
$\sigma_{\mathfrak{a}_{i}}\in\text{SL}_{2}\left(\mathbb{R}\right)$
such that 
\[
\sigma_{\mathfrak{a}_{i}}^{-1}\gamma_{\mathfrak{a}_{i}}\sigma_{\mathfrak{a}_{i}}=\begin{pmatrix}1 & 1\\
0 & 1
\end{pmatrix}.
\]
 $\sigma_{\mathfrak{a}_{i}}$ is determined up to right multiplication
by a translation. We choose $\sigma_{\mathfrak{a}_{i}}$ so that for
each $l\geqslant1$, the semi-strip 
\[
P\left(l\right)\stackrel{\text{def}}{=}\{z\in\mathbb{H}\mid0<x<1,y\geqslant l\},
\]
is mapped into $\mathcal{F}$ by $\sigma_{\mathfrak{a}_{i}}$. 
\begin{defn}
\label{def:compact part}For $i=1,\dots,n$ and $l\geqslant1$, we
define
\begin{align*}
D_{\mathfrak{a}_{i}}\left(l\right) & \stackrel{\text{def}}{=}\sigma_{\mathfrak{a}_{i}}P\left(l\right),
\end{align*}
and 
\[
D\left(l\right)\stackrel{\text{def}}{=}\mathcal{F}\backslash\bigsqcup_{i=1}^{n}D_{\mathfrak{a}_{i}}\left(l\right).
\]

$D_{\mathfrak{a}_{i}}\left(l\right)$ is the part of the fundamental
domain in the $i$th cusp bounded below by the length $\frac{1}{l}$
horocycle and $D\left(l\right)$ is a pre-compact region of $\mathcal{F}$.
By e.g. \cite[Lemma 4.4.6]{Bu92}, the cusps $D_{\mathfrak{a_{i}}}\left(1\right)$
are pairwise disjoint and since $l\geqslant1$, $D_{\mathfrak{a}_{i}}\left(l\right)\cap D_{\mathfrak{a}_{j}}\left(l\right)=\emptyset$
for $i\neq j$ and we can partition the fundamental domain as 
\[
\mathcal{F=}D\left(l\right)\sqcup\bigsqcup_{i=1}^{n}D_{\mathfrak{a}_{i}}\left(l\right).
\]
\end{defn}

\subsection{Test functions\label{subsec:Test-functions}}

In this subsection we introduce the family of test functions used
in Theorem \ref{Theorem: Main analytic Theorem }.
\begin{prop}
\label{prop:g_1}There exists an $f_{1}\in C_{c}^{\infty}\left(\mathbb{R}\right)$
with 
\begin{enumerate}
\item $\textup{Supp}(f_{1})=(-1,1)$.
\item $f_{1}$ is non-negative and even.
\item The Fourier transform $\hat{f_{1}}$ satisfies $\hat{f_{1}}(\xi)\geqslant0$
for $\xi\in\mathbb{R}\cup i\mathbb{R}$.
\item $f_{1}$ is non-increasing in $[0,1).$
\end{enumerate}
\end{prop}

Proposition \ref{prop:g_1} is based on \cite[Section 2.2]{MNP20},
with the extra condition $(4)$ for convenience later on. 
\begin{proof}[Proof of Proposition \ref{prop:g_1}]
Let $\psi_{0}$ be an even, $C^{\infty}$, real valued non-negative
function whose support is exactly $(-\frac{1}{2},\frac{1}{2})$ which
is non-increasing in $[0,\frac{1}{2}).$ Let 
\[
f_{1}(x)\stackrel{\text{def}}{=}\int_{\mathbb{R}}\psi_{0}(x-t)\psi_{0}(t)dt.
\]
It is proved in \cite[Section 2.2]{MNP20} that $f_{1}$ satisfies
$(1)-(3)$. It remains to check $(4)$. Since $f_{1}$ is even we
have $f_{1}'(0)=0$. If $0<x<\frac{1}{2}$, one can calculate that
\begin{align*}
f_{1}'(x) & =\int_{0}^{\frac{1}{2}-x}\left(\psi_{0}(x-z)-\psi_{0}(x+z)\right)\psi_{0}'(z)dz+\int_{\frac{1}{2}-x}^{\frac{1}{2}}\psi_{0}(x-z)\psi_{0}'(z)dz.
\end{align*}
Since $\psi_{0}$ is positive, even and non-increasing in $[0,\frac{1}{2})$,
we have $\psi_{0}'(z)\leqslant0$ and $\psi_{0}(x-z)-\psi_{0}(x+z)\geqslant0$
for all $0\leqslant z\leqslant\frac{1}{2}-x$, so the first integrand
is non-positive. The second integrand is also non-positive since $\psi_{0}$
is non-negative. Therefore $f_{1}'(x)\leqslant0$ in $[0,\frac{1}{2}).$
If $\frac{1}{2}\leqslant x<1$, then 
\[
f_{1}'(x)=\int_{x-\frac{1}{2}}^{\frac{1}{2}}\psi_{0}'(t)\psi_{0}(x-t)dt\leqslant0,
\]
and $f_{1}$ is non-increasing in $[0,1).$ 
\end{proof}
From here on in, we fix such a function $f_{1}$. For $T>1$ we define
\begin{equation}
f_{T}(x)\stackrel{\text{def}}{=}f_{1}\left(\frac{x}{T}\right).\label{Definition of g_T}
\end{equation}
Then by Proposition \ref{prop:g_1}, for each $T>1$, $f_{T}$ is
a non-negative, even, smooth function with support exactly $(-T,T)$
whose Fourier transform $\hat{f}_{T}$ is non-negative on $\mathbb{R}\cup i\mathbb{R}$.
We also have that $f_{T}$ is non-increasing in $[0,T)$.

Let $k_{T}$ denote the inverse Abel transform of $f_{T}$, i.e.
\begin{equation}
k_{T}\left(\rho\right)\stackrel{\text{def}}{=}\frac{-1}{\sqrt{2}\pi}\int_{\rho}^{\infty}\frac{f_{T}'(u)}{\sqrt{\cosh u-\cosh\rho}}du,\label{Definition of k_0}
\end{equation}
which is well defined since $f_{T}$ is compactly supported. We see
that $k_{T}$ is smooth, $\text{Supp}\left(k_{T}\right)\subseteq[0,T)$
and since $f_{T}$ is non-increasing in $[0,T)$, $k_{T}$ is non-negative. 

We now have a fixed family of test functions $f_{T}$ for $T>1$.
We conclude this subsection by stating a lower bound on $\hat{f}_{T}$
in $i\mathbb{R}$ from \cite{MNP20}. 
\begin{lem}[{\cite[Lemma 2.4]{MNP20}}]
\label{lem:Upper bound on g_T}For any $\ep>0$ there exists a constant
$C_{\ep}>0$ such that for all $t\in\mathbb{R}_{\geqslant0}$ and
for all $T>1$ the Fourier transform $\hat{f}_{T}$ satisfies
\begin{equation}
\hat{f}_{T}(it)\geqslant TC_{\ep}e^{T(1-\ep)t}.\label{eq:5}
\end{equation}
\end{lem}

\cite[Lemma 2.4]{MNP20} applies for any function satisfying properties
$(1)-(3)$ from Proposition \ref{prop:g_1} so it also applies here.
Lemma \ref{lem:Upper bound on g_T} tells us that small values of
$\lambda_{1}$ imply large values of $\hat{f}_{T}\left(i\sqrt{\frac{1}{4}-\lambda_{1}}\right)$.

\subsection{Eigenfunction estimates}

Now we have a family of test functions, we proceed with the proof
of Theorem \ref{Theorem: Main analytic Theorem }. For $z,w\in\mathbb{H}$,
$T>1$ we define
\[
k_{T}(z,w)\stackrel{\text{def}}{=}k_{T}\left(d(z,w)\right).
\]
Let $r:[0,\infty)\to\text{\ensuremath{\mathbb{C}}}$ be the function
given by
\[
r(x)=\begin{cases}
i\sqrt{\frac{1}{4}-x} & \text{if \ensuremath{0\leqslant x\leqslant\frac{1}{4}}, }\\
\sqrt{x-\frac{1}{4}} & \text{if \ensuremath{x>\frac{1}{4}}. }
\end{cases}
\]
Let $u_{j}\in L^{2}(X)$ denote the normalized eigenfunction of the
Laplacian on $X$ corresponding to the eigenvalue $\lambda_{j}$.
Our starting point is the following.
\begin{lem}[{Pre-trace inequality \cite[Proposition 5.2]{Ga02}}]
\label{Pre-trace inequality}For all $T>1$ and $z\in\H$ we have
that 
\begin{equation}
\sum_{\substack{j:\text{}\lambda_{j}<\frac{1}{4}}
}\hat{f}_{T}\left(r\left(\lambda_{j}\right)\right)|u_{j}(z)|^{2}\leqslant\sum_{\gamma\in\Gamma_{X}}k_{T}\left(z,\gamma z\right).\label{eq:Pre-trace inequality}
\end{equation}
\end{lem}

Lemma \ref{Pre-trace inequality} is immediately deduced from \cite[Proposition 5.2]{Ga02},
using the fact that $\hat{f}_{T}$ is non-negative on $\mathbb{R}\cup i[0,\frac{1}{2}]$
(the image of $[0,\infty)$ under $r$). We refer to the left hand
side of (\ref{eq:Pre-trace inequality}) as the spectral side and
the right hand side as the geometric side. We prove Theorem \ref{Main analytic theorem}
by integrating (\ref{eq:Pre-trace inequality}). We cannot integrate
(\ref{eq:Pre-trace inequality}) over the full fundamental domain
as the contribution of the parabolic elements
\[
\sum_{\{\text{}\gamma\in\Gamma_{X}\backslash\{\text{Id}\}\text{}\mid\text{}|\trace(\gamma)|=2\text{}\}}k_{T}\left(z,\gamma z\right),
\]
is not absolutely integrable over the fundamental domain $\mathcal{F}$.
We get around this by integrating over the region $D(l)$, as defined
in Definition \ref{def:compact part}, with $l=2$ (the choice $l=2$
could be replaced by any fixed $l>1$). This leads to another issue:
we could potentially lose information on the spectral side after integrating.
This could happen if an eigenfunction concentrated outside $D(2)$.
The following lemma resolves this issue. From now on we write $D=D(2)$.
\begin{lem}[{\cite[Lemma 4.1]{Ga02}}]
 \label{lem:Collarlemmacusps}For any $\kappa>0$, there is a constant
$c\left(\kappa\right)>0$ such that for any $u_{j}$ with $\lambda_{j}\leqslant\frac{1}{4}-\kappa$,
we have
\[
\int_{D}|u_{j}(z)|^{2}d\mu(z)\geqslant c\left(\kappa\right).
\]
The constant $c$ does not depend on the surface $X$. 
\end{lem}

The upshot is that when we integrate (\ref{eq:Pre-trace inequality})
over $D$, we obtain something bounded on the geometric side and we
get a definite contribution from each eigenvalue on the spectral side. 
\begin{rem}
\cite[Lemma 4.1]{Ga02} is stated for quotients of $\mathbb{H}$ by
geometrically finite subgroups of $\text{SL}_{2}(\mathbb{Z})$. The
proof extends trivially to all finite-area non-compact surfaces, as
noted in \cite[Footnote 10]{Ga02}. 
\end{rem}

\subsection{Proof of Theorem \ref{Theorem: Main analytic Theorem }}

We conclude this section by proving Theorem \ref{Theorem: Main analytic Theorem }.
\begin{proof}[Proof of Theorem \ref{Theorem: Main analytic Theorem }]
 Recall that $X$ is a finite-area non-compact hyperbolic surface
with genus $g$, $n=o\left(g^{\frac{1}{2}}\right)$ cusps. We write
$\lambda_{j}=\lambda_{j}\left(X\right)$ and recall that $X$ has
first non-zero Laplacian eigenvalue $\lambda_{1}\leqslant\frac{3}{16}$.
Let $T=4\log g$. By Lemma \ref{Pre-trace inequality}, 
\begin{equation}
\sum_{j:\text{}\lambda_{j}<\frac{1}{4}}\hat{f}_{T}\left(r\left(\lambda_{j}\right)\right)|u_{j}(z)|^{2}\leqslant\sum_{\gamma\in\Gamma_{X}}k_{T}\left(z,\gamma z\right).\label{eq:Pretraceineq}
\end{equation}
Since $\hat{f}_{T}$ is non-negative on $i\mathbb{R}$, $\hat{f}_{T}\circ r$
is non-negative on $[0,\frac{1}{4}]$ and (\ref{eq:Pretraceineq})
still holds if we reduce the sum to just $\lambda_{0}$ and $\lambda_{1}$.
Integrating (\ref{eq:Pretraceineq}) over $D$, we get 
\begin{align}
\hat{f}_{T}\left(r\left(\lambda_{0}\right)\right)\int_{D}|u_{0}(z)|^{2}d\mu(z)+\hat{f}_{T}\left(r\left(\lambda_{1}\right)\right)\int_{D}|u_{1}(z)|^{2}d\mu(z) & \leqslant\int_{D}\sum_{\gamma\in\Gamma_{X}}k_{T}\left(z,\gamma z\right)d\mu(z).\label{eq: Integrated pre-trace}
\end{align}
First we look at the spectral side. The eigenvalue $\lambda_{0}=0$
corresponds to the constant eigenfunction 
\[
u_{0}\left(z\right)=\frac{1}{\sqrt{\text{Vol}(X)}}.
\]
We have
\begin{align*}
\hat{f}_{T}\left(r\left(\lambda_{0}\right)\right)\int_{D}|u_{0}(z)|^{2}d\mu(z) & =\frac{\text{Vol}\left(D\right)}{\text{Vol}\left(X\right)}\hat{f}_{T}\left(\frac{i}{2}\right).
\end{align*}
Recall that
\[
D=\mathcal{F}\backslash\bigsqcup_{i=1}^{n}D_{\mathfrak{a}_{i}}\left(2\right).
\]
Since $D_{\mathfrak{a}_{i}}\left(2\right)$ is isometric to $\{z\in\mathbb{H}\mid0<x<1,y\geqslant2\}$,
$\text{Vol}\left(D_{\mathfrak{a}_{i}}(2)\right)=\frac{1}{2}$ for
each $i$. By Gauss-Bonnet, $\text{Vol}(X)=2\pi\left(2g-2+n\right)$
and we see that
\[
\frac{\text{Vol}(D)}{\text{Vol}(X)}=\frac{2\pi\left(2g-2+n\right)-\frac{n}{2}}{2\pi\left(2g-2+n\right)}=1+O\left(\frac{n}{g}\right).
\]
For the contribution of $\text{\ensuremath{\lambda_{1}}},$ by Lemma
\ref{lem:Collarlemmacusps} with $\kappa=\frac{1}{16}$, there is
a constant $c>0$ with
\begin{equation}
\hat{f}_{T}\left(r\left(\lambda_{1}\right)\right)\int_{D}|u_{1}(z)|^{2}d\mu(z)\geqslant c\hat{f}_{T}\left(r\left(\lambda_{1}\right)\right).\label{eq:Applied collar lemma}
\end{equation}
Let $\varepsilon>0$ be given, then since $\lambda_{1}\leqslant\frac{3}{16}$,
$r(\lambda_{1})=i\sqrt{\frac{1}{4}-\lambda_{1}}$, then by Lemma \ref{lem:Upper bound on g_T},
there is a constant $C_{\varepsilon}>0$ with 
\begin{equation}
\hat{f}_{T}\left(r\left(\lambda_{1}\right)\right)\geqslant TC_{\ep}e^{T(1-\ep)\sqrt{\frac{1}{4}-\lambda_{1}}}.\label{eq:Lower bound on f^}
\end{equation}
Combining (\ref{eq: Integrated pre-trace}), (\ref{eq:Applied collar lemma})
and (\ref{eq:Lower bound on f^}), we see there exists a constant
$C(\ep)>0$ with
\begin{align}
TC(\ep)e^{T(1-\ep)\sqrt{\frac{1}{4}-\lambda_{1}}}+\left(1+O\left(\frac{n}{g}\right)\right)\hat{f}_{T}\left(\frac{i}{2}\right) & \leqslant\int_{D}\sum_{\gamma\in\Gamma_{X}}k_{T}\left(z,\gamma z\right)d\mu(z).\label{eq:spectral side evaluated}
\end{align}
We now look at the geometric side. We arrange the sum in the geometric
side into the contribution from the identity, parabolic and hyperbolic
elements to obtain
\begin{align*}
\int_{D}\sum_{\gamma\in\Gamma_{X}}k_{T}\left(z,\gamma z\right)d\mu(z)= & \sum_{\gamma\in\Gamma_{X}}\int_{D}k_{T}\left(z,\gamma z\right)d\mu(z)\\
= & \int_{D}k_{T}\left(z,z\right)d\mu(z)+\sum_{\{\gamma\in\Gamma_{X}\mid|\trace(\gamma)|>2\}}\int_{D}k_{T}\left(z,\gamma z\right)d\mu(z)\\
 & +\sum_{\{\text{}\gamma\in\Gamma_{X}\backslash\{\text{Id}\}\text{}\mid\text{}|\trace(\gamma)|=2\text{}\}}\int_{D}k_{T}\left(z,\gamma z\right)d\mu(z).
\end{align*}
Interchanging summation and integration is justified since $D$ is
a compact region and $k_{T}$ is supported in $[0,T)$, then for each
$z\in D$, $\#\{\gamma\in\Gamma_{X}\mid d(z,\gamma z)<T\}$ is finite
and the summation is over finitely many terms.

First we treat the contribution of the identity. Since $k_{T}(z,w)=k_{T}\left(d(z,w)\right),$
\begin{align*}
\int_{D}k_{T}\left(z,z\right)d\mu(z) & =\text{Vol}(D)k_{T}(0).
\end{align*}
A calculation involving the Abel Transform, see for example the proof
of \cite[Theorem 9.5.3]{Bu92}, gives that
\[
k_{T}(0)=\frac{1}{4\pi}\int_{-\infty}^{\infty}r\hat{f}_{T}(r)\tanh(\pi r)dr.
\]
We calculate
\begin{align*}
\int_{-\infty}^{\infty}r\hat{f}_{T}(r)\tanh(\pi r)dr & =T\int_{-\infty}^{\infty}r\hat{f}_{1}\left(Tr\right)\tanh(\pi r)dr\\
 & =\frac{1}{T}\int_{-\infty}^{\infty}r'\hat{f}_{1}\left(r'\right)\tanh\left(\frac{\pi r'}{T}\right)dr'\\
 & \leqslant\frac{2}{T}\int_{0}^{\infty}r'\hat{f}_{1}\left(r'\right)dr'\ll\frac{1}{T},
\end{align*}
where the last line follows from the fact that $f_{1}$ is compactly
supported, thus $\hat{f}_{1}$ is a Schwartz function and decays faster
that the inverse of any polynomial. Since $\text{Vol}(D)=2\pi\left(2g-2+n\right)-\frac{n}{2}$,
and $X$ has $o\left(g^{\frac{1}{2}}\right)$ cusps, this tells us
that 
\begin{equation}
\int_{D}k_{T}\left(z,z\right)d\mu(z)=O\left(g\right).\label{eq:Identity contribution}
\end{equation}
Now we look at the hyperbolic terms. By the non-negativity of $k_{T}$,
\[
\sum_{\{\text{}\gamma\in\Gamma_{X}\text{}\mid\text{}|\trace(\gamma)|>2\text{}\}}\int_{D}k_{T}\left(z,\gamma z\right)d\mu(z)\leqslant\sum_{\{\text{}\gamma\in\Gamma_{X}\text{}\mid\text{}|\trace(\gamma)|>2\text{}\}}\int_{\mathcal{F}}k_{T}\left(z,\gamma z\right)d\mu(z).
\]
By arranging the sum into conjugacy classes and unfolding the integral,
one can compute that 
\begin{equation}
\sum_{\{\text{}\gamma\in\Gamma_{X}\text{}\mid\text{}|\trace(\gamma)|>2\text{}\}}\int_{\mathcal{F}}k_{T}\left(z,\gamma z\right)d\mu(z)=\sum_{\gamma\in\mathcal{P}(X)}\sum_{k=1}^{\infty}\frac{l_{\gamma}\left(X\right)}{2\sinh\left(\frac{kl_{\gamma}(x)}{2}\right)}f\left(kl_{\gamma}\left(X\right)\right).\label{eq:Hyperbolic contribution}
\end{equation}
This computation is carried out in detail in \cite[Section 10.2]{Iw02}. 

It remains to bound the contribution of the parabolic elements. Any
$\gamma\in\Gamma_{X}\backslash\{\text{Id}\}$ with $|\trace(\gamma)|=2$
is conjugate to $\gamma_{\mathfrak{a}_{i}}^{l}$ for some unique pair
$i\in\{1,\dots,n\}$ and $l\in\mathbb{Z}\backslash\{0\}$. Since the
centralizer of $\gamma_{\mathfrak{a}_{i}}^{l}$ in $\Gamma_{X}$ is
$\Gamma_{\mathfrak{a}_{i}}$, we see
\[
\sum_{\{\text{}\gamma\in\Gamma_{X}\backslash\{\text{Id}\}\text{}\mid\text{}|\trace(\gamma)|=2\text{}\}}\int_{D}k_{T}\left(z,\gamma z\right)d\mu(z)=\sum_{i=1}^{n}\sum_{l\in\mathbb{Z}^{*}}\sum_{\tau\in\Gamma_{\mathfrak{a}_{i}}\backslash\Gamma}\int_{D}k_{T}\left(z,\tau^{-1}\gamma_{\mathfrak{a}_{i}}^{l}\tau z\right)d\mu(z).
\]
Since $k_{T}$ and $d\mu$ are invariant under isometries, by unfolding
the integral, denoting $\Gamma\cdot D\eqdf\cup_{\gamma\in\Gamma}\gamma D$,
we calculate
\[
\sum_{\tau\in\Gamma_{\mathfrak{a}_{i}}\backslash\Gamma}\int_{D}k_{T}\left(z,\tau^{-1}\gamma_{\mathfrak{a}_{i}}^{l}\tau z\right)d\mu(z)=\int_{\Gamma_{\mathfrak{a}_{i}}\backslash\Gamma\cdot D}k_{T}\left(z,\gamma_{\mathfrak{a}_{i}}^{l}z\right)d\mu(z).
\]
We can choose a fundamental domain $\tilde{\mathcal{F}_{i}}$ for
the action of $\Gamma_{\mathfrak{a}_{i}}$ on $\Gamma\cdot D$ so
that
\begin{align*}
\tilde{\mathcal{F}_{i}} & \subseteq\sigma_{\mathfrak{a}_{i}}\{z\in\mathbb{H}\mid0<x\leqslant1,0<y\leqslant2\},
\end{align*}
and we see, recalling that $\sigma_{\mathfrak{a}_{i}}^{-1}\gamma_{\mathfrak{a}_{i}}\sigma_{\mathfrak{a}_{i}}\left(z\right)=z+1$,

\begin{align*}
\sum_{\tau\in\Gamma_{\mathfrak{a}_{i}}\backslash\Gamma}\int_{D}k_{T}(z,\tau^{-1}\gamma_{\mathfrak{a}_{i}}^{l}\tau z)d\mu(z) & =\int_{\tilde{\mathcal{F}_{i}}}k_{T}(z,\gamma_{\mathfrak{a}_{i}}^{l}z)d\mu(z)\\
 & =\int_{\sigma_{\mathfrak{a}_{i}}^{-1}\left(\tilde{\mathcal{F}_{i}}\right)}k_{T}(z,z+l)d\mu(z)\\
 & \leqslant\int_{x=0}^{x=1}\int_{y=0}^{y=2}k_{T}(z,z+l)d\mu(z).
\end{align*}
We sum over the parabolic conjugacy classes to calculate, 
\begin{align}
\sum_{\{\text{}\gamma\in\Gamma_{X}\backslash\{\text{Id}\}\text{}\mid\text{}|\trace(\gamma)|=2\text{}\}}\int_{D}k_{T}(z,\gamma z)d\mu(z) & \leqslant n\sum_{l\in\mathbb{Z^{*}}}\int_{0}^{1}\int_{0}^{2}k_{T}(z,z+l)d\mu(z)\nonumber \\
 & =n\sum_{l\in\mathbb{Z^{*}}}\int_{0}^{2}k_{T}\left(\text{arcosh}\left(1+\frac{l^{2}}{2y^{2}}\right)\right)y^{-2}dy\nonumber \\
 & =n\sum_{l\in\mathbb{N}}\frac{\sqrt{2}}{l}\int_{\min\left\{ \text{arcosh}\left(1+\frac{l^{2}}{8}\right),T\right\} }^{T}\frac{k_{T}(\rho)\sinh(\rho)}{\sqrt{\cosh(\rho)-1}}d\rho.\label{eq:summation}
\end{align}
On the second line we used that $\cosh d\left(z,z+l\right)=1+\frac{l^{2}}{2y^{2}}$
and on the third line we used the change of variables $\rho=\text{arcosh}\left(1+\frac{l^{2}}{2y^{2}}\right)$
and that $\text{Supp}\left(k_{T}\right)\subseteq[0,T)$. When $\text{arcosh}\left(1+\frac{l^{2}}{8}\right)\leqslant T$,
we use that $f_{T}$ is the Abel transform of $k_{T}$ to see that
\begin{align*}
\int_{\min\left\{ \text{arcosh}\left(1+\frac{l^{2}}{8}\right),T\right\} }^{T}\frac{k_{T}(\rho)\sinh(\rho)}{\sqrt{\cosh(\rho)-1}}d\rho & \leqslant\int_{0}^{T}\frac{k_{T}(\rho)\sinh(\rho)}{\sqrt{\cosh(\rho)-1}}d\rho=f_{T}\left(0\right)=f_{1}\left(0\right).
\end{align*}
If $\text{arcosh}\left(1+\frac{l^{2}}{8}\right)\leqslant T$ then
the contribution to the sum (\ref{eq:summation}) is $0$ and we conclude
that 

\begin{align*}
\sum_{\{\text{}\gamma\in\Gamma_{X}\backslash\{\text{Id}\}\text{}\mid\text{}|\trace(\gamma)|=2\text{}\}}\int_{D}k_{T}(z,\gamma z)d\mu(z) & \leqslant2nf_{1}(0)\sum_{l=1}^{\lfloor\sqrt{8\cosh T}\rfloor}\frac{1}{l}\leqslant2nf_{1}(0)\log\left(2\sqrt{2}e^{\frac{T}{2}}\right).
\end{align*}

Thus combining (\ref{eq:spectral side evaluated}), (\ref{eq:Identity contribution}),
(\ref{eq:Hyperbolic contribution}) and \label{eq:Parabolic contribution}(\ref{eq:Parabolic contribution}),
we conclude that 
\begin{align*}
 & TC(\ep)e^{T(1-\ep)\sqrt{\frac{1}{4}-\lambda_{1}}}+\left(1+O\left(\frac{n}{g}\right)\right)\hat{f}_{T}\left(\frac{i}{2}\right)\\
\leqslant & \sum_{\gamma\in\mathcal{P}(X)}\sum_{k=1}^{\infty}\frac{l_{\gamma}(X)}{2\sinh\left(\frac{kl_{\gamma}(X)}{2}\right)}f_{T}\left(kl_{\gamma}(X)\right)+2nf_{1}(0)\log\left(2\sqrt{2}e^{\frac{T}{2}}\right)+O\left(g\right).
\end{align*}
Recalling that $T=4\log g$, since $f_{T}$ is even,
\begin{align*}
\hat{f}_{T}\left(\frac{i}{2}\right) & =\int_{0}^{\infty}2\cosh\left(\frac{x}{2}\right)f_{T}(x)dx=O\left(g^{2}\right),
\end{align*}
and we deduce that 
\[
C(\varepsilon)\log\left(g\right)g^{4(1-\varepsilon)\sqrt{\frac{1}{4}-\lambda_{1}}}\leqslant\sum_{\gamma\in\mathcal{P}(X)}\sum_{k=1}^{\infty}\frac{l_{\gamma}\left(X\right)}{2\sinh\left(\frac{kl_{\gamma}(x)}{2}\right)}f_{T}\left(kl_{\gamma}\left(X\right)\right)-\hat{f}_{T}\left(\frac{i}{2}\right)+O\left(ng\right),
\]
as claimed.
\end{proof}
\begin{rem}
\label{rem:positivity of random variable}By considering only the
zero eigenvalue, the proof of Theorem \ref{Theorem: Main analytic Theorem }
gives that there exists a constant $\nu\geqslant0$ such that for
sufficiently large $g$ and for any $X\in\mathcal{M}_{g,n}$, 
\[
\sum_{\gamma\in\mathcal{P}(X)}\sum_{k=1}^{\infty}\frac{l_{\gamma}\left(X\right)}{2\sinh\left(\frac{kl_{\gamma}(x)}{2}\right)}f_{T}\left(kl_{\gamma}\left(X\right)\right)-\hat{f}_{T}\left(\frac{i}{2}\right)+\nu ng\geqslant0.
\]
This fact will be important in Section \ref{sec:Proof-of-Theorem}
when we want to apply Markov's inequality to the above quantity, viewed
as a random variable on $\mathcal{M}_{g,n}$.
\end{rem}

\section{Geometric background\label{sec:Geometric-background}}

In this section we shall introduce the necessary background on moduli
space, the Weil-Petersson metric and Mirzakhani's integration formula.
A detailed account of the material in this section can be found in
\cite{Wr20}.

\subsection{Moduli space}

Let $S_{g,n}$ denote an oriented topological surface with genus $g$
and $n$ labeled punctures where $2g-2+n\geqslant1$ and $n\geqslant0$.
A marked surface of signature $\left(g,n\right)$ is a pair $\left(X,\varphi\right)$
where $X$ is a hyperbolic surface with genus $g$ and $n$ cusps
and $\varphi:S_{g,n}\to X$ is a homeomorphism. The Teichmüller space,
denoted by $\mathcal{T}_{g,n}$, is defined by 
\[
\mathcal{T}_{g,n}\stackrel{\text{def}}{=}\{\text{\text{Marked surfaces} }(X,\varphi)\}/\sim,
\]
where $\left(X_{1},\varphi_{1}\right)\sim\left(X_{2},\varphi_{2}\right)$
if there exists an isometry $m:X_{1}\to X_{2}$ such that $\varphi_{2}$
and $m\circ\varphi_{1}$ are isotopic. Let $\text{Homeo}^{+}\left(S_{g,n}\right)$
denote the group of orientation preserving homeomorphisms of $S_{g,n}$
which do not permute the punctures and let $\text{Homeo}_{0}^{+}\left(S_{g,n}\right)$
denote the subgroup of homeomorphisms isotopic to the identity. The
mapping class group is defined as 
\[
\text{MCG}_{g,n}\stackrel{\text{def}}{=}\text{Homeo}^{+}\left(S_{g,n}\right)/\text{Homeo}_{0}^{+}\left(S_{g,n}\right).
\]
$\text{Homeo}^{+}\left(S_{g,n}\right)$ acts on $\mathcal{T}_{g,n}$
by precomposition of the marking and $\text{Homeo}^{+}\left(S_{g,n}\right)$
acts trivially hence $\text{MCG}_{g,n}$ acts on $\mathcal{T}_{g,n}$
and we define the moduli space by 
\[
\mathcal{M}_{g,n}\stackrel{\text{def}}{=}\mathcal{T}_{g,n}/\text{MCG}_{g,n}.
\]
$\mathcal{M}_{g,n}$ can be thought of as the set of equivalence classes
of genus $g$ hyperbolic surfaces with $n$ labeled cusps where two
surfaces are equivalent if they are isometric by an isometry which
preserves the labeling of the cusps. 

Given $\underline{l}\in\mathbb{R}_{\geqslant0}^{n}$, in a similar
way, we define $\mathcal{T}_{g,n}\left(\underline{l}\right)$ as the
Teichmüller space of genus $g$ hyperbolic surfaces with labeled geodesic
boundary components $\left(b_{1},...,b_{n}\right)$ with lengths $\left(l_{1},...,l_{n}\right)$.
We allow $l_{i}=0$, then the boundary component $b_{i}$ is a replaced
by a cusp and we recover 
\[
\mathcal{M}_{g,n}=\mathcal{M}_{g,n}\left(0,...,0\right).
\]

\subsection{Weil-Petersson metric}

The space $\mathcal{T}_{g,n}\left(\underline{l}\right)$ carries a
natural symplectic structure known as the Weil-Petersson symplectic
form and is denoted by $\omega_{WP}$. It is invariant under the action
of the mapping class group and descends to a symplectic form on the
quotient $\mathcal{M}_{g,n}\left(\underline{l}\right)$. The form
$\omega_{WP}$ induces the volume form
\[
\text{dVol}_{WP}\stackrel{\text{def}}{=}\frac{1}{\left(3g-3+n\right)!}\bigwedge_{i=1}^{3g-3+n}\omega_{WP},
\]
which is also invariant under the action of the mapping class group
and descends to a volume form on $\mathcal{M}_{g,n}\left(\underline{l}\right)$.
The quantity $3g-3+n$ appears as the dimension of the Teichmüller
and moduli space. We write $dX$ as shorthand for $\text{dVol}_{WP}$.
We let $V_{g,n}\left(\underline{l}\right)$ denote $\text{Vol}_{WP}\left(\mathcal{M}_{g,n}\left(\underline{l}\right)\right)$,
the total volume of $\mathcal{M}_{g,n}\left(\underline{l}\right)$,
which is finite. We write $V_{g,n}$ to denote $V_{g,n}\left(\underline{0}\right)$.

As in \cite{Mi13,WX21,LW21}, we define a probability measure on $\mathcal{M}_{g,n}$
by normalizing $\text{dVol}_{WP}$. Indeed, for any Borel subset $\mathcal{B\subseteq\mathcal{M}}_{g,n}$,
\[
\mathbb{P}_{WP}^{g,n}\left[\mathcal{B}\right]\stackrel{\text{def}}{=}\frac{1}{V_{g,n}}\int_{\mathcal{M}_{g,n}}\boldsymbol{1}_{\mathcal{B}}dX,
\]
where 
\[
\boldsymbol{1}_{\mathcal{B}}\left(X\right)=\begin{cases}
0 & \text{if \ensuremath{x\notin\mathcal{B}}, }\\
1 & \text{if \ensuremath{x\in\mathcal{B}}. }
\end{cases}
\]
is the indicator function on $\mathcal{B}$. We write $\mathbb{E}_{WP}^{g,n}$
to denote expectation with respect to $\mathbb{P}_{WP}^{g,n}$.

\subsection{Mirzakhani's integration formula}

We recall Mirzakhani's integration formula from \cite{Mi07}. We define
a multi-curve to be an ordered $k$-tuple $\left(\gamma_{1},...,\gamma_{k}\right)$
of disjoint, simple, non-peripheral closed curves. Let $\Gamma=\left[\gamma_{1},...,\gamma_{k}\right]$
denote the homotopy class of a multi-curve. The mapping class group
$\text{MCG}_{g,n}$ acts naturally on homotopy classes of multi-curves
and we denote the orbit containing $\Gamma$ by
\[
\mathcal{O}_{\Gamma}=\left\{ \left(g\cdot\gamma_{1},...,g\cdot\gamma_{k}\right)\mid g\in\text{MCG}_{g,n}\right\} .
\]
Given a function $F:\mathbb{R}_{\geqslant0}^{k}\to\mathbb{R}_{\geqslant0}$,
define $F^{\Gamma}:\mathcal{M}_{g,n}\to\mathbb{R}$ by 
\[
F^{\Gamma}\left(X\right)=\sum_{\left(\alpha_{1},...,\alpha_{k}\right)\in\mathcal{O}_{\Gamma}}F\left(l_{\alpha_{1}}\left(X\right),...,l_{\alpha_{k}}\left(X\right)\right),
\]
where $l_{\alpha_{i}}\left(X\right)$ is defined for $\left(X,\varphi\right)\in\mathcal{T}_{g,n}$
as the length of the geodesic in the homotopy class of $\varphi\left(\alpha_{i}\right)$.
Note that the function $F^{\Gamma}$ is well defined on $\mathcal{M}_{g,n}$
since we are summing over the orbit $\mathcal{O}_{\Gamma}$. Let $S_{g,n}\left(\Gamma\right)$
denote the result of cutting the surface $S_{g,n}$ along $\left(\gamma_{1},...,\gamma_{k}\right)$,
then $S_{g,n}\left(\Gamma\right)=\sqcup_{i=1}^{s}S_{g_{i},n_{i}}$
for some $\left\{ \left(g_{i},n_{i}\right)\right\} _{i=1}^{s}$. Each
$\gamma_{i}$ gives rise to two boundary components $\gamma_{i}^{1}$
and $\gamma_{i}^{2}$ of $S_{g,n}\left(\Gamma\right)$. Given $\underline{x}=\left(x_{1},...,x_{k}\right)\in\mathbb{R}_{\geqslant0}^{k}$,
let $\mathcal{M}\left(S_{g,n}\left(\Gamma\right);l_{\Gamma}=\underline{x}\right)$
be the moduli space of hyperbolic surfaces homeomorphic to $S_{g,n}\left(\Gamma\right)$
such that for $1\leqslant i\leqslant k$, $l_{\gamma_{i}^{1}}=l_{\gamma_{i}^{2}}=x_{i}.$
Let $\underline{x}^{(i)}$ denote the tuple of coordinates $x_{j}$
of $\underline{x}$ such that $\gamma_{j}$ is a boundary component
of $S_{g_{i},n_{i}}$. We have that 
\[
\mathcal{M}\left(S_{g,n}\left(\Gamma\right);l_{\Gamma}=\underline{x}\right)=\prod_{i=1}^{s}\mathcal{M}_{g_{i},n_{i}}\left(\underline{x}^{(i)}\right),
\]
and we define 
\begin{align*}
V_{g,n}\left(\Gamma,\underline{x}\right)\stackrel{\text{def}}{=} & \text{ Vol}_{WP}\left(\mathcal{M}\left(S_{g,n}\left(\Gamma\right);l_{\Gamma}=\underline{x}\right)\right)=\prod_{i=1}^{s}V_{g_{i},n_{i}}\left(\underline{x}^{(i)}\right).
\end{align*}
In terms of the above notation we have the following.
\begin{thm}[{Mirzakhani's Integration Formula \cite[Theorem 7.1]{Mi07}}]
\label{thm:MIF} Given $\Gamma=\left[\gamma_{1},...,\gamma_{k}\right]$,
\[
\int_{\mathcal{M}_{g,n}}F^{\Gamma}\left(X\right)dX=C_{\Gamma}\int_{\mathbb{R}_{\geqslant0}^{k}}F\left(x_{1},...,x_{k}\right)V_{g,n}\left(\Gamma,\underline{x}\right)x_{1}\cdots x_{k}dx_{1}\cdots dx_{k},
\]
where the constant $C_{\Gamma}\in(0,1]$ only depends on $\Gamma$.
Moreover, if $g>2$ and $\Gamma=\left[\gamma\right]$ where $\gamma$
is a simple, non-separating closed curve, then $C_{\Gamma}=\frac{1}{2}.$
\end{thm}

\section{Geometric estimates\label{sec:Geometric-estimates}}

Recall that the family of test functions $f_{T}$ in Theorem \ref{Theorem: Main analytic Theorem }
is defined in (\ref{Definition of g_T}) with $T=4\log g$. For $X\in\mathcal{M}_{g,n}$,
$\gamma\in\mathcal{P}(X)$, $k\in\mathbb{N}$, we shall denote 
\[
H_{X,k}(\gamma)\stackrel{\text{def}}{=}\frac{l_{\gamma}\left(X\right)}{2\sinh\left(\frac{kl_{\gamma}(x)}{2}\right)}f_{T}\left(kl_{\gamma}\left(X\right)\right).
\]
The goal of this section is to prove the following.
\begin{thm}
\label{Main Geometric Theorem}For $0\leqslant\alpha<\frac{1}{2}$,
let $n=O\left(g^{\alpha}\right)$. For any $\ep_{1}>0$ there exists
a constant $c_{1}\left(\ep_{1}\right)>0$, independent of $\alpha$,
with
\[
\mathbb{E}_{WP}^{g,n}\left[\sum_{\gamma\in\mathcal{P}(X)}\sum_{k=1}^{\infty}H_{X,k}(\gamma)-\hat{f_{T}}\left(\frac{i}{2}\right)\right]\ll n^{2}g+\log\left(g\right)^{5}\cdot g+c_{1}\left(\ep_{1}\right)\left(\log g\right)^{\beta+1}\cdot n^{2}\cdot g^{1+4\epsilon_{1}},
\]
where $\beta>0$ is a universal constant.
\end{thm}

Throughout Section \ref{sec:Geometric-estimates} we shall always
have $n=O\left(g^{\alpha}\right)$ for fixed $0\leqslant\alpha<\frac{1}{2}$. 
\begin{rem}
The proof of Theorem \ref{Main Geometric Theorem} closely follows
\cite[Chapters 6 \& 7]{WX21}, making the necessary adaptations to
the case of surfaces with cusps. We therefore omit some arguments
that are identical in the compact and non-compact case and instead
refer the reader to the relevant place. 
\end{rem}

\subsection{Method}

We prove Theorem \ref{Main Geometric Theorem} by considering separately
the contribution of different types of geodesics. As in \cite{WX21},
we introduce the following notation. 
\begin{defn}
For $X\in\mathcal{M}_{g,n}$ we define 
\begin{enumerate}
\item $\mathcal{P}_{sep}^{s}(X)\stackrel{\text{def}}{=}\{\gamma\in\mathcal{P}(X)\mid\gamma\text{ is simple and separating}\}$.
\item $\mathcal{P}_{nsep}^{s}(X)\stackrel{\text{def}}{=}\{\gamma\in\mathcal{P}(X)\mid\gamma\text{ is simple and non-separating}\}$.
\item $\mathcal{P}^{ns}(X)\stackrel{\text{def}}{=}\{\gamma\in\mathcal{P}(X)\mid\gamma\text{ is non-simple}\}$.
\end{enumerate}
Notice that $\mathcal{P}(X)=\mathcal{P}_{sep}^{s}(X)\sqcup\mathcal{P}_{nsep}^{s}(X)\sqcup\mathcal{P}^{ns}(X)$.
We partition the sum $\sum_{\gamma\in\mathcal{P}(X)}\sum_{k=1}^{\infty}H_{X,k}(\gamma)$
as
\begin{align*}
\sum_{\gamma\in\mathcal{P}(X)}\sum_{k=1}^{\infty}H_{X,k}(\gamma)= & \sum_{\gamma\in\mathcal{P}(X)}H_{X,1}(\gamma)+\sum_{\gamma\in\mathcal{P}(X)}\sum_{k=2}^{\infty}H_{X,k}(\gamma)\\
= & \sum_{\gamma\in\mathcal{P}_{sep}^{s}(X)}H_{X,1}(\gamma)+\sum_{\gamma\in\mathcal{P}_{nsep}^{s}(X)}H_{X,1}(\gamma)+\sum_{\gamma\in\mathcal{P}^{ns}(X)}H_{X,1}(\gamma)\\
 & +\sum_{\gamma\in\mathcal{P}(X)}\sum_{k=2}^{\infty}H_{X,k}(\gamma).
\end{align*}
Subtracting $\hat{f}\left(\frac{i}{2}\right)$ and taking Weil-Petersson
expectations, we see 
\begin{align}
 & \mathbb{E}_{WP}^{g,n}\left[\sum_{\gamma\in\mathcal{P}(X)}\sum_{k=1}^{\infty}H_{X,k}(\gamma)-\hat{f_{T}}\left(\frac{i}{2}\right)\right]\nonumber \\
\leqslant & \underbrace{\mathbb{E}_{WP}^{g,n}\left[\sum_{\gamma\in\mathcal{P}_{sep}^{s}(X)}H_{X,1}(\gamma)\right]}_{(a)}+\underbrace{\left|\mathbb{E}_{WP}^{g,n}\left[\sum_{\gamma\in\mathcal{P}_{nsep}^{s}(X)}H_{X,1}(\gamma)\right]-\hat{f}\left(\frac{i}{2}\right)\right|}_{(b)}\nonumber \\
 & +\underbrace{\mathbb{E}_{WP}^{g,n}\left[\sum_{\gamma\in\mathcal{P}(X)}\sum_{k=2}^{\infty}H_{X,k}(\gamma)\right]}_{(c)}+\underbrace{\mathbb{E}_{WP}^{g,n}\left[\sum_{\gamma\in\mathcal{P}^{ns}(X)}H_{X,1}(\gamma)\right]}_{(d)}.\label{eq:Labelled formula}
\end{align}
\end{defn}

The remainder of this section is dedicated to bounding terms $(a)-(d)$,
from which Theorem \ref{Main Geometric Theorem} will follow. 
\begin{itemize}
\item Since terms $(a)$ and $(b)$ depend on simple geodesics, we can bound
them by applying Mirzakhani's integration formula directly. 
\item To bound $(c)$ we consider geodesics with length $<1$ and length
$\geqslant1$ separately. The contribution of geodesics with length
$\geqslant1$ can be bounded deterministically. Any geodesic with
length $<1$ must be simple, by e.g. \cite[Theorem 4.2.4]{Bu92},
so we can apply Mirzakhani's integration formula directly to bound
their contribution. 
\item To bound $(d)$, we cannot apply Mirzakhani's integration formula
directly since the geodesics are not simple. Instead, we pass from
non-simple geodesics to subsurfaces with simple geodesic boundary
and apply Mirzakhani's integration formula to the simple boundary
geodesics.
\end{itemize}

\subsection{Contribution of simple separating geodesics}

In this subsection we bound term $(a)$ in (\ref{eq:Labelled formula}),
the contribution of simple separating geodesics. In particular, we
prove the following.
\begin{lem}
\label{Lema contribution of simple and separating}
\[
\mathbb{E}_{WP}^{g,n}\left[\sum_{\gamma\in\mathcal{P}_{sep}^{s}(X)}H_{X,1}(\gamma)\right]\ll n^{2}g.
\]
\end{lem}

\begin{proof}
We have 
\begin{equation}
\mathbb{E}_{WP}^{g,n}\left[\sum_{\gamma\in\mathcal{P}_{sep}^{s}(X)}H_{X,1}(\gamma)\right]=\frac{1}{V_{g,n}}\int_{\mathcal{M}_{g,n}}\sum_{\gamma\in\mathcal{P}_{sep}^{s}(X)}H_{X,1}(\gamma)dX.\label{eq:simple sep 1}
\end{equation}
We shall apply Mirzakhani's integration formula, Theorem \ref{thm:MIF},
to bound the integral in (\ref{eq:simple sep 1}). Recall that $S_{g,n}$
is a topological surface with genus $g$ and $n$ labeled punctures.
For $0\leqslant i\leqslant\lfloor\frac{g}{2}\rfloor,$ $0\leqslant j\leqslant n,$
let $\alpha_{i,j}$ be a simple closed curve in $S_{g,n}$ which separates
$S_{g,n}$ into subsurfaces $S_{i,j+1}$ and $S_{g-i,n-j+1}$, each
with one boundary component and $j$ and $n-j$ punctures respectively.
Then $\alpha_{i,j}$ partitions the punctures into two disjoint subsets
$I$ and $J$ of size $j$ and $n-j$ respectively. Let $\left[\alpha_{i,j}\right]$
denote the homotopy class of $\alpha_{i,j}$. The orbit $\text{MCG}_{g,n}\cdot\left[\alpha_{i,j}\right]$
is determined by the set $\left\{ \left(i,j+1,I\right),\left(g-i,n-j+1,J\right)\right\} $,
since the mapping class group does not permute the punctures. Therefore
given $i$ and $j$, there are ${n \choose j}$ $\text{MCG}_{g,n}$-orbits
of simple separating closed curves on $S_{g,n}$ which separate off
a subsurface with genus $i$ and with $j$ punctures. Recalling that
\[
H_{X,1}(\gamma)=\frac{l_{\gamma}\left(X\right)}{2\sinh\left(\frac{l_{\gamma}(X)}{2}\right)}f_{T}\left(l_{\gamma}\left(X\right)\right),
\]
 we now apply Mirzakhani's integration formula, Theorem \ref{thm:MIF},
to see 
\begin{align*}
 & \frac{1}{V_{g,n}}\int_{\mathcal{M}_{g,n}}\sum_{\gamma\in\mathcal{P}_{sep}^{s}(X)}H_{X,1}(\gamma)dX\\
\leqslant & \sum_{\substack{0\leqslant i\leqslant g,\text{}0\leqslant j\leqslant n\\
2\leqslant2i+j\leqslant2g+n-2
}
}\int_{0}^{\infty}{n \choose j}\frac{x^{2}}{\sinh\left(\frac{x}{2}\right)}f_{T}(x)\frac{V_{i,j+1}\left(\underline{0}_{j},x\right)V_{g-i,n-j+1}\left(\underline{0}_{n-j},x\right)}{V_{g,n}}dx.
\end{align*}
By Lemma \ref{lem:Sinh bound-1},
\begin{align*}
V_{a,b}(\underline{0}_{b-1},x) & \leqslant\frac{2\sinh\left(\frac{x}{2}\right)}{x}V_{a,b},
\end{align*}
giving 
\begin{align*}
 & \frac{1}{V_{g,n}}\int_{\mathcal{M}_{g,n}}\sum_{\gamma\in\mathcal{P}_{sep}^{s}(X)}H_{X,1}(\gamma)dX\\
\leqslant & \frac{4}{V_{g,n}}\left(\sum_{\substack{0\leqslant i\leqslant g,\text{}0\leqslant j\leqslant n\\
2\leqslant2i+j\leqslant2g+n-2
}
}{n \choose j}\cdot\frac{V_{i,j+1}V_{g-i,n-j+1}}{V_{g,n}}\right)\int_{0}^{\infty}\sinh\left(\frac{x}{2}\right)f_{T}(x)dx.
\end{align*}
Since $f_{T}$ is bounded independently of $T$ and supported in $[0,T)$,
we see 
\[
\int_{0}^{\infty}\sinh\left(\frac{x}{2}\right)f_{T}(x)dx\ll e^{\frac{T}{2}}.
\]
By Lemma \ref{Summing over},
\[
\sum_{\substack{0\leqslant i\leqslant g,\text{}0\leqslant j\leqslant n\\
2\leqslant2i+j\leqslant2g+n-2
}
}\frac{n!}{j!\left(n-j\right)!}\cdot\frac{V_{i,j+1}V_{g-i,n-j+1}}{V_{g,n}}\ll\frac{n^{2}}{g},
\]
giving 
\begin{align*}
\mathbb{E}_{WP}^{g,n}\left[\sum_{\gamma\in\mathcal{P}_{sep}^{s}(X)}H_{X,1}(\gamma)\right] & \ll\frac{n^{2}}{g}\cdot e^{\frac{T}{2}}\ll n^{2}g,
\end{align*}
as claimed.
\end{proof}

\subsection{Contribution of simple non-separating geodesics}

In this subsection we deal with the contribution of simple non-separating
geodesics (term $(b)$ in (\ref{eq:Labelled formula})). We prove
the following. 
\begin{lem}
\label{contribution of cimple and non-separating}
\[
\left|\mathbb{E}_{WP}^{g,n}\left[\sum_{\gamma\in\mathcal{P}_{nsep}^{s}(X)}H_{X,1}(\gamma)\right]-\hat{f}\left(\frac{i}{2}\right)\right|\ll n^{2}g+n\cdot\log\left(g\right)^{2}\cdot g.
\]
\end{lem}

\begin{proof}
Let $\alpha_{0}$ be an unoriented simple non-separating closed curve
in $S_{g,n}$. There is just one $\text{MCG}_{g,n}$-orbit of simple
non-separating closed curves on $S_{g,n}$ and we have 
\[
\sum_{\gamma\in\mathcal{P}_{nsep}^{s}(X)}H_{X,1}(\gamma)dX=2\sum_{\gamma\in\text{MCG}_{g,n}\cdot\alpha_{0}}H_{X,1}(\gamma),
\]
where the factor of $2$ occurs since geodesics in $\mathcal{P}\left(X\right)$
are oriented. Applying Mirzakhani's integration formula, we get 
\begin{align*}
\int_{\mathcal{M}_{g,n}}\sum_{\gamma\in\mathcal{P}_{nsep}^{s}(X)}H_{X,1}(\gamma)dX & =\frac{1}{2}\int_{0}^{\infty}\frac{x^{2}}{\sinh(\frac{x}{2})}f_{T}(x)V_{g-1,n+2}\left(\underline{0}_{n},x,x\right)dx,
\end{align*}
where the factor $\frac{1}{2}$ occurs since $\alpha_{0}$ is simple
and non-separating, c.f. Theorem \ref{thm:MIF}. By Theorem \ref{Variable n asymptotic},
\[
V_{g-1,n+2}=V_{g,n}\cdot\left(1+O\left(\frac{n^{2}}{g}\right)\right).
\]
Then we have, by applying Lemma \ref{lem:Sinh bound-1},
\[
\frac{V_{g-1,n+2}\left(\underline{0}_{n},x,x\right)}{V_{g,n}}=\left(\frac{2\sinh\frac{x}{2}}{x}\right)^{2}\left(1+O\left(\frac{n^{2}+nx^{2}}{g}\right)\right).
\]
This gives 
\begin{align*}
 & \frac{1}{V_{g,n}}\int_{\mathcal{M}_{g,n}}\sum_{\gamma\in\mathcal{P}_{nsep}^{s}(X)}\frac{l_{\gamma}(X)}{\sinh\left(\frac{l_{\gamma}(X)}{2}\right)}f_{T}\left(l_{\gamma}\left(X\right)\right)dX\\
 & =\int_{0}^{T}2\sinh\left(\frac{x}{2}\right)f_{T}(x)\left(1+O\left(\frac{n^{2}+nx^{2}}{g}\right)\right)dx.
\end{align*}
Since $\hat{f_{T}}\left(\frac{i}{2}\right)$ is even,
\[
\hat{f_{T}}\left(\frac{i}{2}\right)=\int_{0}^{T}2\cosh\left(\frac{x}{2}\right)f_{T}(x)dx,
\]
and we have 
\begin{align*}
 & \left|\mathbb{E}_{WP}^{g,n}\left[\sum_{\gamma\in\mathcal{P}_{nsep}^{s}(X)}H_{X,1}(\gamma)\right]-\hat{f_{T}}\left(\frac{i}{2}\right)\right|\\
= & \left|\int_{0}^{T}2\sinh\left(\frac{x}{2}\right)f_{T}(x)\left(1+O\left(\frac{1+n^{2}+nx^{2}}{g}\right)\right)dx-\int_{0}^{T}2\cosh\left(\frac{x}{2}\right)f_{T}(x)dx\right|\\
\ll & \left|\int_{0}^{T}2\left(\sinh\left(\frac{x}{2}\right)-\cosh\left(\frac{x}{2}\right)\right)\cdot f_{T}(x)dx\right|+\left|\int_{0}^{T}2\sinh\left(\frac{x}{2}\right)f_{T}(x)\left(\frac{n^{2}+nx^{2}}{g}\right)dx\right|.
\end{align*}
Using that $2\left(\cosh\left(\frac{x}{2}\right)-\sinh\left(\frac{x}{2}\right)\right)=e^{-x}$,
\[
\left|\int_{0}^{T}2\left(\sinh\left(\frac{x}{2}\right)-\cosh\left(\frac{x}{2}\right)\right)\cdot f_{T}(x)dx\right|\ll1.
\]
Recalling $T=4\log g$, we calculate 
\begin{align*}
\left|\int_{0}^{T}2\sinh\left(\frac{x}{2}\right)f_{T}(x)\left(\frac{1+n^{2}+n^{2}x}{g}\right)dx\right| & \ll\frac{e^{\frac{T}{2}}\left(n^{2}+nT^{2}\right)}{g}\ll n^{2}g+n\cdot\log\left(g\right)^{2}\cdot g,
\end{align*}
and 
\[
\left|\mathbb{E}_{WP}^{g,n}\left[\sum_{\gamma\in\mathcal{P}_{nsep}^{s}(X)}H_{X,1}(\gamma)\right]-\hat{f_{T}}\left(\frac{i}{2}\right)\right|\ll n^{2}g+n\cdot\log\left(g\right)^{2}\cdot g,
\]
as claimed.
\end{proof}

\subsection{Iterates of primitive geodesics}

We now look at the contribution of iterates of primitive geodesics
(term $(c)$ in (\ref{eq:Labelled formula})). The aim of this subsection
is to prove the following.
\begin{lem}
\label{lemma- contribution of iterates}
\[
\mathbb{E}_{WP}^{g,n}\left[\sum_{\gamma\in\mathcal{P}(X)}\sum_{k=2}^{\infty}H_{X,k}(\gamma)\right]\ll\log\left(g\right)^{2}\cdot g.
\]
\end{lem}

In order to prove Lemma \ref{lemma- contribution of iterates}, we
need the following soft geodesic counting bound.
\begin{lem}
For any $X\in\mathcal{M}_{g,n}$ and any $L>0$ we have 
\[
\#\{\gamma\in\mathcal{P}\left(X\right)\mid1\leqslant l_{\gamma}\left(X\right)\leqslant L\}\ll ge^{L}.
\]
\end{lem}

\begin{proof}
Let $\#_{0}(X,L)$ denote the number of closed geodesics on $X$ with
length $\leqslant L$ which are not iterates of closed geodesics of
length $\leqslant2\text{arcsinh}(1)$. An immediate adaptation of
the proof of \cite[Lemma 6.6.4]{Bu92} using the non-compact version
of the Collar Theorem (\cite[Lemma 4.4.6]{Bu92}) tells us that 
\[
\#_{0}(X,L)\leqslant\left(g-1+\frac{n}{2}\right)e^{L+6}.
\]
 \cite[Lemma 4.4.6]{Bu92} also tells us that the number of primitive
geodesics on $X$ with length $\leqslant4\text{arcsinh}(1)$ is bounded
above by $3g-3+n$. Using that $n=o\left(\sqrt{g}\right)$, we conclude
that 
\begin{align*}
\#\{\gamma\in\mathcal{P}\left(X\right)\mid1\leqslant l_{\gamma}\left(X\right)\leqslant L\}\leqslant & \left(g-1+\frac{n}{2}\right)e^{L+6}+3g-3+n\ll ge^{L},
\end{align*}
as claimed.
\end{proof}
We now proceed with the proof of Lemma \ref{lemma- contribution of iterates}.
\begin{proof}[Proof of Lemma \ref{lemma- contribution of iterates}]
Let $X\in\mathcal{M}_{g,n}$. We write 
\[
\sum_{\gamma\in\mathcal{P}(X)}\sum_{k=2}^{\infty}H_{X,k}(\gamma)=\sum_{\{\gamma\in\mathcal{P}(X)\mid l_{\gamma}(X)<1\}}\sum_{k=2}^{\infty}H_{X,k}(\gamma)+\sum_{\{\gamma\in\mathcal{P}(X)\mid l_{\gamma}(X)\geqslant1\}}\sum_{k=2}^{\infty}H_{X,k}(\gamma).
\]
 By Lemma \ref{lemma- contribution of iterates},
\[
\#\{\gamma\in\mathcal{P}\left(X\right)\mid1\leqslant l_{\gamma}\left(X\right)\leqslant L\}\ll ge^{L}.
\]
 We then have
\begin{align*}
\sum_{\{\gamma\in\mathcal{P}(X)\mid l_{\gamma}(X)\geqslant1\}}\sum_{k=2}^{\infty}H_{X,k}(\gamma) & \ll\sum_{\{\gamma\in\mathcal{P}(X)\mid1\leqslant l_{\gamma}(X)\leqslant\frac{T}{2}\}}l_{\gamma}\left(X\right)e^{-l_{\gamma}\left(X\right)}\\
 & \leqslant\sum_{m=1}^{\lfloor\frac{T}{2}\rfloor}me^{-m}\cdot\#\{\gamma\in\mathcal{P}\left(X\right)\mid m\leqslant l_{\gamma}\left(X\right)\leqslant m+1\}\\
 & \ll g\sum_{m=1}^{\lfloor\frac{T}{2}\rfloor}m\ll\left(\log g\right)^{2}\cdot g.
\end{align*}
Taking Weil-Petersson expectations, we see 
\begin{align}
\mathbb{E}_{WP}^{g,n}\left[\sum_{\gamma\in\mathcal{P}(X)}\sum_{k=2}^{\infty}H_{X,k}(\gamma)\right]= & \mathbb{E}_{WP}^{g,n}\left[\sum_{\{\gamma\in\mathcal{P}(X)\mid l_{\gamma}(X)<1\}}\sum_{k=2}^{\infty}H_{X,k}(\gamma)\right]\nonumber \\
 & +O\left(\left(\log g\right)^{2}g\right).\label{eq:Non primitive eq1}
\end{align}
For each $\gamma\in\mathcal{P}(X)$,
\begin{align*}
H_{X,k}(\gamma) & =\frac{l_{\gamma}\left(X\right)}{2\sinh\left(\frac{kl_{\gamma}(x)}{2}\right)}f_{T}\left(kl_{\gamma}\left(X\right)\right)\leqslant f\left(0\right),
\end{align*}
and if $k\geqslant\frac{T}{l_{\gamma}(X)}$ then $f_{T}\left(kl_{\gamma}\left(X\right)\right)=0$.
This tells us that

\begin{equation}
\mathbb{E}_{WP}^{g,n}\left[\sum_{\{\gamma\in\mathcal{P}(X)\mid l_{\gamma}(X)<1\}}\sum_{k=2}^{\infty}H_{X,k}(\gamma)\right]\leqslant f(0)\cdot T\cdot\mathbb{E}_{WP}^{g,n}\left[\sum_{\{\gamma\in\mathcal{P}(X)\mid l_{\gamma}(X)<1\}}\frac{1}{l_{\gamma}(X)}\right].\label{non primitive eq 2}
\end{equation}
It remains to bound 
\[
\mathbb{E}_{WP}^{g,n}\left[\sum_{\{\gamma\in\mathcal{P}(X)\mid l_{\gamma}(X)<1\}}\frac{1}{l_{\gamma}(X)}\right].
\]
Any geodesic $\gamma\in\mathcal{P}(X)$ with length $l_{\gamma}(X)\leqslant1<4\text{arcsinh}1$
must be simple by e.g. \cite[Theorem 4.2.4]{Bu92}. Therefore we can
apply Mirzakhani's integration formula to get 
\begin{align}
\mathbb{E}_{WP}^{g,n}\left[\sum_{\{\gamma\in\mathcal{P}(X)\mid l_{\gamma}(X)<1\}}\frac{1}{l_{\gamma}(X)}\right] & \leqslant\frac{1}{V_{g,n}}\int_{0}^{1}V_{g-1,n+2}(\underline{0}_{n},t,t)dt\nonumber \\
 & +\sum_{\substack{0\leqslant i\leqslant g,\text{}0\leqslant j\leqslant n\\
2\leqslant2i+j\leqslant2g+n-2
}
}\frac{n!}{j!\left(n-j\right)!}\cdot\frac{V_{i,j+1}V_{g-i,n-j+1}}{V_{g,n}}\nonumber \\
 & \ll\frac{V_{g-1,n+2}}{V_{g,n}}+\frac{n^{2}}{g}\ll1,\label{non iterate eq 3}
\end{align}
where on the last line we applied Lemma \ref{Summing over} and Theorem
\ref{Variable n asymptotic}. Thus combining (\ref{eq:Non primitive eq1}),
(\ref{non primitive eq 2}) and (\ref{non iterate eq 3}) we see 
\[
\mathbb{E}_{WP}^{g,n}\left[\sum_{\gamma\in\mathcal{P}(X)}\sum_{k=2}^{\infty}H_{X,k}(\gamma)\right]\ll\left(\log g\right)^{2}\cdot g,
\]
as required.
\end{proof}

\subsection{Non-simple geodesics}

We now need to deal with the contribution of the non-simple primitive
geodesics, (term $(d)$ in (\ref{eq:Labelled formula})). In this
subsection we shall prove the following 
\begin{lem}
\label{contribution of non-simple}There is a constant $\beta_{1}>0$
such that for any $\ep_{1}>0$ there is a constant $c_{1}\left(\ep_{1}\right)>0$
such that 
\[
\mathbb{E}_{WP}^{g,n}\left[\sum_{\gamma\in\mathcal{P}^{ns}(X)}H_{X,1}(\gamma)\right]\ll\left(\log g\right)^{6}\cdot g+c_{1}\left(\ep_{1}\right)\left(\log g\right)^{\beta_{1}}\cdot n^{2}\cdot g^{1+4\ep_{1}}.
\]
\end{lem}

We prove Lemma \ref{contribution of non-simple} through a sequence
of lemmas. Before we give a brief outline of the method, we need the
concept of a filling closed curve.
\begin{defn}
Let $X$ be a finite-area hyperbolic surface with possible boundary.
A closed curve $\eta\subset Y$ is filling if the complement $Y\backslash\eta$
is a disjoint union of disks and cylinders such that every cylinder
either deformation retracts to a boundary component of $Y$ or is
a neighbourhood of a cusp. We let $\#_{\text{fill}}(X,L)$ denote
the number of oriented filling geodesics on $X$ with lengths $\leqslant L$.
\end{defn}

\subsection*{Idea of the proof of Lemma \ref{contribution of non-simple}}

We shall extend the method of \cite[Section 7]{WX21} to non-compact
surfaces. The basic idea is as follows.
\begin{itemize}
\item Given a surface $X\in\mathcal{M}_{g,n}$ and a geodesic $\gamma\in\mathcal{P}^{ns}\left(X\right)$,
we construct a subsurface $X(\gamma)$ of $X$ with geodesic boundary
(of controlled length) which is filled by $\gamma$. The multiplicity
of the map $\gamma\mapsto X(\gamma)$ is bounded by the number of
filling geodesics of $X(\gamma)$. This allows us to write 
\[
\sum_{\gamma\in\mathcal{P}^{ns}(X)}H_{X,1}(\gamma)\leqslant\sum_{\substack{Y\text{ subsurface of }X\\
Y\text{ has geodesic boundary}
}
}\sum_{\text{filling geodesics \ensuremath{\gamma} on }Y}H_{X,1}\left(l_{X}\left(\gamma\right)\right).
\]
\item We control the length of a filling geodesic in terms of $l_{X}\left(\partial Y\right)$
in Lemma \ref{Proposition- SUbsurface construction} and apply \cite[Theorem 4]{WX21}
to bound the number of filling geodesics on a subsurface and show
that there is an explicit function $A$, supported in $[0,2T)$, with
\[
\sum_{\gamma\in\mathcal{P}^{ns}(X)}H_{X,1}(\gamma)\leqslant\sum_{\substack{Y\text{ subsurface of }X\\
Y\text{ has geodesic boundary}
}
}A\left(l_{X}\left(\partial Y\right)\right).
\]
\item Since the boundary of each subsurface $Y$ consists of simple closed
geodesics, we can apply Mirzakhani's integration formula to bound
the Weil-Petersson expectation of 
\[
\sum_{\{Y\text{subsurface of }X\text{ with geodesic boundary}\}}A\left(l_{X}\left(\partial Y\right)\right).
\]
\end{itemize}
\begin{defn}
Let $X\in\mathcal{M}_{g,n}$ be a hyperbolic surface and let $\gamma\subset X$
be a non-simple closed geodesic. Let $N_{\delta}(\gamma)$ denote
the $\delta$-neighborhood of $\gamma$ where $\delta$ is sufficiently
small to ensure that $N_{\delta}(\gamma)$ deformation retracts to
$\gamma$ and that the boundary $\partial N_{\delta}(\gamma)$ is
a disjoint union of simple closed curves. We define $X(\gamma)$ to
be the connected subsurface obtained from $N_{\delta}(\gamma)$ as
follows: for each boundary component $\xi\in N_{\delta}(\gamma)$,
\begin{itemize}
\item If $\xi$ bounds a disc we fill the disc into $N_{\delta}(\gamma)$.
\item If $\xi$ is homotopically non-trivial we shrink it to the unique
simple closed geodesics in its free homotopy class and deform $N_{\delta}(\gamma)$
accordingly.
\item If two different components $\xi,\xi'$ deform to the same geodesic
then we do not glue them together, we view $X(\gamma)$ as an open
subsurface of $X$.
\item If $\xi$ is freely homotopic to a closed horocycle bounding a cusp
$C_{i}$ we fill the cusp into $N_{\delta}(\gamma)$.
\end{itemize}
After deforming $N_{\delta}(\gamma)$ in this way we obtain the surface
$X(\gamma).$

The construction of $X(\gamma)$ allows us to control $\text{Vol}\left(X(\gamma)\right)$
and the length of $\partial X(\gamma)$ in terms of $l_{\gamma}\left(X\right)$,
as summarized by the following lemma. Bounding $\text{Vol}\left(X(\gamma)\right)$
corresponds to bounding the Euler characteristic of $X(\gamma)$ by
Gauss-Bonnet.
\end{defn}

\begin{lem}
\label{Proposition- SUbsurface construction}Let $X\in\mathcal{M}_{g,n}$
and $\gamma$ be a non-simple closed geodesic on $X$. The subsurface
$X(\gamma)$ of $X$ satisfies 
\end{lem}

\begin{enumerate}
\item $\gamma$ is a filling geodesic of $X(\gamma)$.
\item The length of the boundary satisfies 
\[
l\left(\partial X(\gamma)\right)\leqslant2l_{\gamma}(X).
\]
\item The volume satisfies 
\[
\text{Vol}\left(X(\gamma)\right)\leqslant4l_{\gamma}(X).
\]
\end{enumerate}
Lemma \ref{Proposition- SUbsurface construction} is proved in \cite[Proposition 47]{NWX20}
for compact surfaces. The proof in our case is identical. This leads
us to make the following definition.
\begin{defn}
With $T=4\log g$, $X\in\mathcal{M}_{g,n}$, we define 
\[
\text{Sub}(X)\stackrel{\text{def}}{=}\{Y\subset X\mid Y\text{ is a connected subsurface of \ensuremath{X} with geodesic boundary\},}
\]
and
\begin{align*}
\text{Sub}_{T}(X) & \stackrel{\text{def}}{=}\{Y\in\text{Sub}(X)\mid l(\partial Y)\leqslant2T,\text{\text{Vol}}(Y)\leqslant4T\},
\end{align*}
where we allow two distinct simple closed geodesics on the boundary
of $Y$ to be a single simple closed geodesic in $X$. 
\end{defn}

Lemma \ref{Proposition- SUbsurface construction} tells us that for
any $X\in\mathcal{M}_{g,n}$, any non-simple geodesic $\gamma$ with
length $\leqslant T$ fills a subsurface $X(\gamma)\in\text{Sub}_{T}(X)$.
If any other $\gamma'\in\mathcal{P}(X)$ satisfies $X(\gamma')=X(\gamma)$
then $\gamma'$ is also a filling geodesic of $X(\gamma)$ with length
$\leqslant T$. We have 
\begin{equation}
\{\gamma'\in\mathcal{P}^{ns}(X)\mid X(\gamma')=X(\gamma)\}\subseteq\{\text{oriented filling geodesics of }X(\gamma)\text{ with length }\leqslant T\}.\label{eq: non-simple geodesics to filling geodesics}
\end{equation}
 Therefore we will need to control the number of non-simple geodesics
which fill a given subsurface. This is achieved by the following theorem.
\begin{thm}[{\cite[Theorem 4]{WX21}}]
\label{Main Counting} Let $m=2g'-2+n'\geqslant1$. For any $\ep_{1}>0$
there exists a constant $c(\ep_{1},m)$ only depending on $\ep_{1}$
and $m$ such that for any $X\in\mathcal{M}_{g',n'}(x_{1},...,x_{n'})$
where $x_{i}\geqslant0$, we have
\[
\#_{\textup{fill}}(X,L)\leqslant c(\epsilon_{1},m)\cdot e^{L-\frac{1-\ep_{1}}{2}\sum_{i=1}^{n}x_{i}}.
\]
\end{thm}

\begin{rem}
\cite[Theorem 4]{WX21} is stated in for surfaces without cusps, i.e.
$x_{i}>0$, however the extension to $x_{i}\geqslant0$ is immediate.
Indeed, \cite[Theorem 4]{WX21} follows from \cite[Theorem 38]{WX21}
and \cite[Lemma 10]{WX21}. \cite[Theorem 38]{WX21} already holds
for non-compact surfaces and it is straightforward to check that the
basic counting result \cite[Lemma 10]{WX21} generalizes to non-compact
surfaces.

We can now pass from non-simple geodesics to subsurfaces with geodesic
boundary. This is done in the following lemma, proved in \cite[Proposition 30]{WX21}
for $X\in\mathcal{M}_{g}$. The proof is identical in our case.
\end{rem}

\begin{lem}
\label{Lemma- Geodesics to subsurfaces}For any $\ep_{1}>0$, $X\in\mathcal{M}_{g,n}$,
there exists a constant $c_{1}\left(\ep_{1}\right)$ only depending
on $\ep_{1}$ such that
\begin{align}
\sum_{\gamma\in\mathcal{P}^{ns}(X)}H_{X,1}(\gamma) & \ll Te^{T}\sum_{\substack{Y\in\textup{\textup{Sub}}_{T}(X)\\
|\chi(Y)|\geqslant34
}
}e^{-\frac{l(\partial Y)}{4}}+c_{1}\left(\ep_{1}\right)T\sum_{\substack{Y\in\textup{\textup{Sub}}_{T}(X)\\
1\leqslant|\chi(Y)|\leqslant33
}
}e^{\frac{T}{2}-\frac{1-\epsilon_{1}}{2}l(\partial Y)}.\label{eq:Geodesics to subsurfaces}
\end{align}
\end{lem}

\begin{rem}
The difference between the first and second term arises because we
apply Theorem \ref{Main Counting} to subsurfaces with $1\leqslant|\chi(Y)|\leqslant34$
whereas we only apply a soft geodesic counting result, $\#_{\textup{fill}}(X,L)\leqslant\text{Area}(X)\cdot e^{L+6}$,
to subsurfaces with $|\chi(Y)|\geqslant34$. The reason for this is
that it is not clear how badly the constant $c(\epsilon_{1},m)$ from
Theorem \ref{Main Counting} depends on the Euler characteristic $m$
so we can only apply Theorem \ref{Main Counting} to subsurfaces with
uniformly bounded Euler characteristic. As a consequence of forthcoming
calculations, the Weil-Petersson expectation of the number of subsurfaces
$Y\in\text{Sub}_{T}(X)$ with $|\chi(Y)|\geqslant k$ is sufficiently
small for any $k\geqslant34$ so that we can accept the loss from
the soft geodesic counting.
\end{rem}

For the remainder of the section, we assume that $g$ is sufficiently
large so that for $Y\in\text{Sub}_{T}(X)$, the map $Y\mapsto\partial Y$
is injective. This is justified since any two distinct subsurfaces
in $Y_{1},Y_{2}\in\text{Sub}_{T}(X)$ with $\partial Y_{1}=\partial Y_{2}$
must satisfy $Y_{1}\cup Y_{2}=X$, giving
\[
\text{Vol}\left(X\right)=2\pi\left(2g-2+n\right)\leqslant\text{Vol}\left(Y_{1}\right)+\text{Vol}\left(Y_{2}\right)\leqslant8T=32\log g,
\]
which is not possible for sufficiently large $g$. 

We now want to apply Mirzakhani's integration formula to bound the
Weil-Petersson expectation of the right hand side of (\ref{eq:Geodesics to subsurfaces}).
We introduce the following notation.
\begin{notation}
Let $X\in\mathcal{M}_{g,n}$. For a subsurface $Y_{0}\in\text{Sub}_{T}(X)$,
we write 
\[
Y_{0}=Y_{0}\left(q,\left(g_{0},a_{0},n_{0}\right),\left\{ \left(g_{1},a_{1},n_{1}\right),\dots,\left(g_{q},a_{q},n_{q}\right)\right\} \right)=Y_{0}\left(q,\underline{g},\underline{a},\underline{n}\right),
\]
 to indicate that $Y_{0}$ has the following properties.
\begin{itemize}
\item $Y_{0}$ is homeomorphic to $S_{g_{0},k+a_{0}}$where $k>0$.
\item $Y_{0}$ has $a_{0}$ cusps and $k$ simple geodesic boundary components.
There are $n_{0}\geqslant0$ pairs of simple geodesics in $Y_{0}$
which correspond to a single simple closed geodesic in $X$.
\item The interior of its complement $X\backslash Y_{0}$ consists of $q\geqslant1$
components $Y_{1},...,Y_{q}$ where $Y_{i}$ is homeomorphic to $S_{g_{i},n_{i}+a_{i}}$.
We observe that $n_{i}\geqslant1$ and 
\begin{lyxlist}{00.00.0000}
\item [{i)}] $\sum_{i=1}^{q}2g_{i}-2+n_{i}+a_{i}=2g-2+n-\left|\chi\left(Y_{0}\right)\right|.$
\item [{ii)}] $\sum_{i=1}^{q}n_{i}=k-2n_{0}.$
\item [{iii)}] $\sum_{j=1}^{q}a_{j}=n-a_{0}.$
\end{lyxlist}
\end{itemize}
\end{notation}

Given $X\in\mathcal{M}_{g,n}$ and a choice of marking, any $Y_{0}(q,\underline{a},\underline{n},\underline{g})\in\text{Sub}_{T}(X)$
is freely homotopic to the image under the marking of a subsurface
$Y\subset S_{g,n}$ where $Y$ is in the $\text{MCG}_{g,n}$-orbit
of a subsurface $\tilde{Y_{0}}=\tilde{Y}_{0}(q,\underline{a},\underline{n},\underline{g})\subset S_{g,n}$
(with $\tilde{Y_{0}}$ homeomorphic to $S_{g_{0},k+a_{0}}$, where
$S_{g,n}\backslash\tilde{Y}_{0}$ has $q$ components $\tilde{Y}_{1},...,\tilde{Y}_{q}$
with $\tilde{Y}_{i}$ homeomorphic to $S_{g_{i},n_{i}+a_{i}}$ with
$n_{i}$ boundary components and $a_{i}$ punctures). We write $\left[\tilde{Y}_{0}\right]$
to denote the homotopy class of $\tilde{Y}_{0}$. Since the mapping
class group does not permute the punctures of $S_{g,n}$, the number
of distinct $\text{MCG}_{g,n}$-orbits of subsurfaces corresponding
to a given choice of $q,\left(g_{0},n_{0},g_{0}\right),\left\{ \left(g_{1},a_{1},n_{1}\right),\dots,\left(g_{q},a_{q},n_{q}\right)\right\} $
is bounded above by 
\[
\frac{n!}{a_{0}!\cdot\cdots\cdot a_{q}!}.
\]

\begin{lem}
\label{Euler characteristic 17+}
\begin{equation}
\mathbb{E}_{WP}^{g,n}\left[\sum_{\substack{Y\in\text{Sub}_{T}(X)\\
|\chi\left(Y\right)|\geqslant34
}
}e^{-\frac{l(\partial Y)}{4}}\right]\ll\frac{\left(\log g\right)^{5}}{g^{3}}.\label{eq:Lemma >34}
\end{equation}
\end{lem}

\begin{proof}
We start by bounding the contribution of a given $\text{MCG}_{g,n}$-orbit
to (\ref{eq:Lemma >34}). Let $g_{0},a_{0}$, $k$ be fixed with $m=2g_{0}-2+k+a_{0}\geqslant34$.
By Gauss-Bonnet, we have that $m\leqslant\frac{4T}{2\pi}\leqslant\frac{5}{2}\log g$.
For $n_{0},n_{1},\dots,n_{q},a_{1},\dots,a_{q},g_{1},\dots,g_{q}\geqslant0$
with $\sum_{i=1}^{q}n_{i}=k-2n_{0}$ and $\sum_{j=1}^{q}a_{j}=n-a_{0}$,
we have 
\begin{align*}
 & \frac{1}{V_{g,n}}\int_{\mathcal{M}_{g,n}}\sum_{[Y]\in\text{MCG}_{g,n}\cdot[\tilde{Y}_{0}(q,\underline{a},\underline{n},\underline{g})]}e^{-\frac{l(\partial Y)}{4}}\boldsymbol{1}_{[0,2T]}\left(l_{X}\left(\partial Y\right)\right)dX\\
= & \frac{1}{V_{g,n}}\int_{\mathcal{M}_{g,n}}\sum_{[\partial Y]\in\text{MCG}_{g,n}\cdot[\partial\tilde{Y}_{0}(q,\underline{a},\underline{n},\underline{g})]}e^{-\frac{l(\partial Y)}{4}}\boldsymbol{1}_{[0,2T]}\left(l_{X}\left(\partial Y\right)\right)dX,
\end{align*}
since the map $Y\mapsto\partial Y$ is injective. By applying Mirzakhani's
integration formula, one can compute that 

\begin{align*}
 & \frac{1}{V_{g,n}}\int_{\mathcal{M}_{g,n}}\sum_{[Y]\in\text{MCG}_{g,n}\cdot[\tilde{Y}_{0}(q,\underline{a},\underline{n},\underline{g})]}e^{-\frac{l(\partial Y)}{4}}\boldsymbol{1}_{[0,2T]}\left(l_{X}\left(\partial Y\right)\right)dX\\
\ll & e^{\frac{7}{2}T}\frac{V_{g_{0},k+a_{0}}V_{g_{1},n_{1}+a_{1}}\cdots V_{g_{q},n_{q}+a_{q}}}{V_{g,n}\cdot n_{0}!n_{1}!\cdots n_{q}!}.
\end{align*}
A near identical computation is carried out in detail in \cite[Proposition 31]{WX21}
so we omit it here. We now sum over the $\text{MCG}_{g,n}$-orbits
to bound the contribution of subsurfaces in $\text{Sub}_{T}\left(X\right)$
with a given Euler characteristic. We calculate 
\begin{align*}
 & \mathbb{E}_{WP}^{g,n}\left[\sum_{\substack{Y\in\text{Sub}_{T}(X)\\
Y\cong S_{g_{0},k+a_{0}}
}
}e^{-\frac{l(\partial Y)}{4}}\right]\\
\leqslant & \sum_{n_{0}=0}^{\lfloor\frac{k}{2}\rfloor}\sum_{q=1}^{k-2a_{0}}\sum_{\mathcal{A}}\frac{1}{V_{g,n}}\cdot{n \choose a_{0},\dots,a_{q}}\cdot\int_{\mathcal{M}_{g,n}}\sum_{[Y]\in\text{MCG}_{g,n}\cdot[\tilde{Y}_{0}(q,\underline{a},\underline{n},\underline{g})]}e^{-\frac{l(\partial Y)}{4}}\boldsymbol{1}_{[0,2T]}\left(l_{X}\left(\partial Y\right)\right)dX\\
\ll & e^{\frac{7}{2}T}\sum_{n_{0}=0}^{\lfloor\frac{k}{2}\rfloor}\sum_{q=1}^{k-2a_{0}}\sum_{\left\{ \left(g_{j},n_{j},q_{j}\right)\right\} _{j=1}^{q}\in\mathcal{A}}{n \choose a_{0},\dots,a_{q}}.\frac{V_{g_{0},k+a_{0}}V_{g_{1},n_{1}+a_{1}}\cdots V_{g_{q},n_{q}+a_{q}}}{V_{g,n}\cdot n_{0}!n_{1}!\cdots n_{q}!},
\end{align*}
where for a given $n_{0}$ and $q$, the summation is over the set
of ``admissible triples'' $\mathcal{A}$, whose elements we denote
by $\left\{ \left(g_{j},n_{j},q_{j}\right)\right\} _{j=1}^{q}$, which
we define to be the set of $\left\{ \left(g_{1},a_{1},n_{1}\right),\dots,\left(g_{q},a_{q},n_{q}\right)\right\} $
where $g_{j},a_{j}\geqslant0$, $n_{j}\geqslant1$ and $2g_{j}+a_{j}+n_{j}\geqslant3$
such that
\begin{lyxlist}{00.00.0000}
\item [{i)}] $\sum_{i=1}^{q}\left(2g_{i}-2+n_{i}+a_{i}\right)=2g-2+n-m$.
\item [{ii)}] $\sum_{i=1}^{q}n_{i}=k-2n_{0}$. 
\item [{iii)}] $\sum_{j=1}^{q}a_{j}=n-a_{0}$. 
\end{lyxlist}
Recalling that $34\leqslant m=2g_{0}-2+k+a_{0}\leqslant\frac{5}{2}\log g$
is fixed, we apply lemma \ref{WpVOLest} to see 
\begin{align*}
 & \sum_{n_{0}=0}^{\lfloor\frac{k}{2}\rfloor}\sum_{q=1}^{k-2n_{0}}\sum_{\left\{ \left(g_{j},n_{j},a_{j}\right)\right\} _{j=1}^{q}\in\mathcal{A}}\frac{n!}{a_{0}!\cdot\cdots\cdot a_{q}!}.\frac{V_{g_{0},k+a_{0}}V_{g_{1},n_{1}+a_{1}}\cdots V_{g_{q},n_{q}+a_{q}}}{V_{g,n}\cdot n_{0}!n_{1}!\cdots n_{q}!}\\
\ll & \sum_{n_{0}=0}^{\lfloor\frac{k}{2}\rfloor}\sum_{q=1}^{k-2n_{0}}\left(2g_{0}+k+a_{0}-3\right)!\cdot\frac{n^{a_{0}}}{g^{m}}\ll\frac{k^{2}\left(2g_{0}+k+a_{0}-3\right)!}{g^{m}}.
\end{align*}
 Summing over the possible values of $g_{0},a_{0}$ and $k$, we calculate
\begin{align*}
 & \mathbb{E}_{WP}^{g,n}\left[\sum_{\substack{Y\in\text{Sub}_{T}(X)\\
|\chi\left(Y\right)|\geqslant34
}
}e^{-\frac{l(\partial Y)}{4}}\right]\\
\ll & e^{\frac{7}{2}T}\sum_{0\leqslant a_{0}\leqslant\lceil\frac{4T}{2\pi}\rceil}\sum_{\substack{1\leqslant k\leqslant\lceil\frac{4T}{2\pi}\rceil+2-a_{0}}
}\sum_{\substack{34\leqslant2g_{0}-2+k+a_{0}\leqslant\lceil\frac{4T}{2\pi}\rceil}
}\mathbb{E}_{WP}^{g,n}\left[\sum_{\substack{Y\in\text{Sub}_{T}(X)\\
Y\cong S_{g_{0},k+a_{0}}
}
}e^{-\frac{l(\partial Y)}{4}}\right]\\
\ll & e^{\frac{7}{2}T}\sum_{0\leqslant a_{0}\leqslant\lceil\frac{4T}{2\pi}\rceil}\sum_{\substack{1\leqslant k\leqslant\lceil\frac{4T}{2\pi}\rceil+2-a_{0}}
}\sum_{\substack{34\leqslant2g_{0}-2+k+a_{0}\leqslant\lceil\frac{4T}{2\pi}\rceil}
}\frac{k^{2}\left(2g_{0}+a_{0}+k-3\right)!n^{a_{0}}}{g^{2g_{0}+a_{0}+k-2}}\\
\ll & T^{5}e^{\frac{7T}{2}}\frac{1}{g^{2g_{0}+\frac{a_{0}}{2}+k-2}}\ll\frac{T^{5}e^{\frac{7T}{2}}}{g^{18}},
\end{align*}
since $2g_{0}+a_{0}+k\geqslant36$ guarantees that $2g_{0}+\frac{a_{0}}{2}+k\geqslant18$.
Recalling that $T=4\log g$, we conclude that 
\[
\mathbb{E}_{WP}^{g,n}\left[\sum_{\substack{Y\in\text{Sub}_{T}(X)\\
|\chi\left(Y\right)|\geqslant34
}
}e^{-\frac{l(\partial Y)}{4}}\right]\ll\frac{\left(\log g\right)^{5}}{g^{3}},
\]
as required.
\end{proof}
\begin{lem}
\label{Non-simple lema 3}There is a constant $\beta>0$ such that
for any $\ep_{1}>0$, 
\[
\mathbb{E}_{WP}^{g,n}\left[\sum_{\substack{Y\in\text{Sub}_{T}(X)\\
1\leqslant|\chi\left(Y\right)|\leqslant33
}
}e^{\frac{T}{2}-\frac{1-\epsilon_{1}}{2}l(\partial Y)}\right]\ll\left(\log g\right)^{\beta}\cdot n^{2}\cdot g^{1+4\ep_{1}}.
\]
\end{lem}

\begin{proof}
Let $\ep_{1}>0$, $g_{0}\geqslant0,a_{0}\geqslant0$ and $k\geqslant1$
be fixed with $1\leqslant m=2g_{0}-2+k+a_{0}\leqslant33$. The computation
in \cite[Proposition 34]{WX21} gives that there exists a fixed $\beta>0$
\footnote{Note the value of $\beta$ in \cite[Proposition 34]{WX21} is $66$
and corresponds to the choice to consider $|\chi(Y)|\leqslant16$
as opposed to our choice of $33$. Here we could for example take
$\text{\ensuremath{\beta<}}135$. Fixed powers of $\log g$ will be
negligible in the final calculations.}with
\begin{align*}
 & \frac{1}{V_{g,n}}\int_{\mathcal{M}_{g,n}}\sum_{\tilde{Y}\in\text{MCG}_{g,n}\cdot\tilde{Y}_{0}(q,\underline{a},\underline{n},\underline{g})}e^{\frac{T}{2}-\frac{1-\ep_{1}}{2}l_{X}\left(\partial\tilde{Y}\right)}\boldsymbol{1}_{[0,2T]}\left(l_{X}\left(\partial\tilde{Y}\right)\right)dX\\
\ll & \frac{T^{\beta}e^{\frac{T}{2}+\ep_{1}T}}{V_{g,n}n_{0}!\cdots n_{q}!}V_{g_{1},n_{1}+a_{1}}\cdots V_{g_{q},n_{q}+a_{q}}.
\end{align*}
Then we see that 
\begin{align*}
 & \mathbb{E}_{WP}^{g,n}\left[\sum_{\substack{Y\in\text{Sub}_{T}(X)\\
Y\cong S_{g_{0},k+a_{0}}
}
}e^{\frac{T}{2}-\frac{1-\ep_{1}}{2}l(\partial Y)}\right]\ll T^{\beta}e^{\frac{T}{2}+\ep_{1}T}\sum_{n_{0}=0}^{\lfloor\frac{k}{2}\rfloor}\sum_{q=1}^{k-2a_{0}}\sum_{\mathcal{A}}\frac{n!}{a_{0}!\cdots a_{q}!}\frac{V_{g_{1},n_{1}+a_{1}}\cdots V_{g_{q},n_{q}+a_{q}}}{n_{0}!\cdots n_{q}!V_{g,n}},
\end{align*}
where, as before, for given $n_{0}$ and $q$ the summation is over
the set $\mathcal{A}$ of ``admissible triples'' $\left\{ \left(g_{j},n_{j},q_{j}\right)\right\} _{j=1}^{q}$
where $g_{j},a_{j}\geqslant0$, $n_{j}\geqslant1$ and $2g_{j}+a_{j}+n_{j}\geqslant3$
such that $\sum_{i=1}^{q}2g_{i}-2+n_{i}+a_{i}=2g-2+n-m$, $\sum_{i=1}^{q}n_{i}=k-2n_{0}$
and $\sum_{j=1}^{q}a_{j}=n-a_{0}$. We apply Lemma \ref{WpVOLest}
to calculate that
\begin{align*}
 & T^{\beta}e^{\frac{T}{2}+\ep_{1}T}\sum_{n_{0}=0}^{\lfloor\frac{k}{2}\rfloor}\sum_{q=1}^{k-2n_{0}}\sum_{\mathcal{A}}\frac{n!}{a_{0}!\cdots a_{q}!}\frac{V_{g_{1},n_{1}+a_{1}}\cdots V_{g_{q},n_{q}+a_{q}}}{n_{0}!\cdots n_{q}!V_{g,n}}\\
\ll & T^{\beta}e^{\frac{T}{2}+\ep_{1}T}\sum_{n_{0}=0}^{\lfloor\frac{k}{2}\rfloor}\sum_{q=1}^{k-2n_{0}}\frac{n^{a_{0}}}{g^{2g_{0}+a_{0}+k-2}}\ll T^{\beta}e^{\frac{T}{2}+\epsilon_{1}T}\frac{n^{a_{0}}}{g^{2g_{0}+a_{0}+k-2}}.
\end{align*}
We sum over possible values of $g_{0},a_{0}$ and $k$ to see that
\begin{align*}
\mathbb{E}_{WP}^{g,n}\left[\sum_{\substack{Y\in\text{Sub}_{T}(X)\\
1\leqslant|\chi\left(Y\right)|\leqslant33
}
}e^{\frac{T}{2}-\frac{1-\epsilon_{1}}{2}l(\partial Y)}\right] & \ll\sum_{\substack{\left(g_{0},a_{0},k\right)\\
3\leqslant2g_{0}+a_{0}+k\leqslant35
}
}T^{\beta}e^{\frac{T}{2}+\ep_{1}T}\frac{n^{a_{0}}}{g^{2g_{0}+a_{0}+k-2}}\\
 & \ll T^{\beta}e^{\frac{T}{2}+\ep_{1}T}\cdot\frac{n^{2}}{g}\ll\left(\log g\right)^{\beta}n^{2}g^{1+4\ep_{1}},
\end{align*}
as claimed. 
\end{proof}
We can now prove Lemma \ref{contribution of non-simple}.
\begin{proof}[Proof of Lemma \ref{contribution of non-simple}]
Combining Lemma \ref{Lemma- Geodesics to subsurfaces}, Lemma \ref{Euler characteristic 17+}
and Lemma \ref{Non-simple lema 3} we deduce that for any $\ep_{1}>0$
there exists a constant $c_{1}\left(\ep_{1}\right)$ such that 
\begin{align*}
 & \mathbb{E}_{WP}^{g,n}\left[\sum_{\gamma\in\mathcal{P}^{ns}(X)}H_{X,1}(\gamma)\right]\\
\ll & e^{T}T\mathbb{E}_{WP}^{g,n}\left[\sum_{\substack{Y\in\text{Sub}_{T}(X)\\
|\chi\left(Y\right)|\geqslant34
}
}e^{-\frac{l(\partial Y)}{4}}\right]+c_{1}\left(\ep_{1}\right)T\mathbb{E}_{WP}^{g,n}\left[\sum_{\substack{Y\in\text{Sub}_{T}(X)\\
1\leqslant|\chi\left(Y\right)|\leqslant33
}
}e^{\frac{T}{2}-\frac{1-\ep_{1}}{2}l(\partial Y)}\right]\\
\ll & \left(\log g\right)^{6}g+c_{1}\left(\ep_{1}\right)\left(\log g\right)^{\beta+1}n^{2}g^{1+4\ep_{1}},
\end{align*}
concluding the proof.
\end{proof}

\subsection{Proof of Theorem \ref{Main Geometric Theorem}}

Finally we conclude the section with the proof of Theorem \ref{Main Geometric Theorem}.
\begin{proof}[Proof of Theorem \ref{Main Geometric Theorem}]
By Lemma \ref{Lema contribution of simple and separating}, Lemma
\ref{contribution of cimple and non-separating}, Lemma \ref{lemma- contribution of iterates}
and Lemma \ref{contribution of non-simple} together with (\ref{eq:Labelled formula})
we see that there is a constant $\beta$ such that for any $\varepsilon_{1}>0$
there exists a constant $c_{1}\left(\ep_{1}\right)$ with
\[
\mathbb{E}_{WP}^{g,n}\left[\sum_{\gamma\in\mathcal{P}(X)}\sum_{k=1}^{\infty}H_{X,k}(\gamma)-\hat{f}_{T}\left(\frac{i}{2}\right)\right]\ll n^{2}g+\log\left(g\right)^{6}g+c_{1}\left(\ep_{1}\right)\left(\log g\right)^{\beta+1}n^{2}g^{1+4\epsilon_{1}}.
\]
\end{proof}

\section{Proof of Theorem \ref{Main Theorem of paper}\label{sec:Proof-of-Theorem}}

We now conclude with the proof of Theorem \ref{Main Theorem of paper}.
\begin{proof}[Proof of Theorem \ref{Main Theorem of paper}]
Let $n=O\left(g^{\alpha}\right)$ for some $0\leqslant\alpha<\frac{1}{2}$
and let $0<\ep<\min\left\{ \frac{1}{4},\frac{1}{2}-\alpha\right\} $
be given. For $X\in\mathcal{M}_{g,n}$, we define 
\[
\tilde{\lambda}_{1}\left(X\right)\eqdf\begin{cases}
\lambda_{1}\left(X\right) & \text{if it exists, }\\
\frac{1}{4} & \text{otherwise. }
\end{cases}
\]
Our aim is to prove that 
\[
\mathbb{P}_{WP}^{g,n}\left[\tilde{\lambda}_{1}\left(X\right)\leqslant\frac{1}{4}-\frac{\left(2\alpha+1\right)^{2}}{16}-\varepsilon\right]\to0,
\]
as $g\to\infty$. By Remark \ref{rem:positivity of random variable},
there exists a constant $\nu\geqslant0$ such that for $g$ sufficiently
large,
\[
\sum_{\gamma\in\mathcal{P}(X)}\sum_{k=1}^{\infty}\frac{l_{\gamma}\left(X\right)}{2\sinh\left(\frac{kl_{\gamma}(x)}{2}\right)}f_{T}\left(kl_{\gamma}\left(X\right)\right)-\hat{f}_{T}\left(\frac{i}{2}\right)+\nu ng\geqslant0,
\]
for any $X\in\mathcal{M}_{g,n}.$ By Theorem \ref{Main Geometric Theorem},
for any $\ep_{1}>0$ there is constant $c_{1}\left(\ep_{1}\right)>0$
with
\begin{align*}
\mathbb{E}_{WP}^{g,n}\left[\sum_{\gamma\in\mathcal{P}(X)}\sum_{k=1}^{\infty}H_{X,k}(\gamma)-\hat{f_{T}}\left(\frac{i}{2}\right)+\nu ng\right]\ll & n^{2}g+\log\left(g\right)^{6}g+c_{1}\left(\ep_{1}\right)\left(\log g\right)^{\beta+1}n^{2}g^{1+4\ep_{1}},
\end{align*}
where $\beta>0$ is a universal constant. Taking $\ep_{1}<\frac{\ep}{8}$
and applying Markov's inequality, 
\[
\mathbb{P}_{WP}^{g,n}\left[\sum_{\gamma\in\mathcal{P}(X)}\sum_{k=1}^{\infty}H_{X,k}(\gamma)-\hat{f}_{T}\left(\frac{i}{2}\right)+\nu ng>n^{2}g^{1+\ep}\right]\ll_{\epsilon}\left(1+\frac{\log\left(g\right)^{6}}{n^{2}}+\left(\log g\right)^{\beta+1}\right)g^{-\frac{\ep}{2}}\cdot
\]
However, if $X\in\mathcal{M}_{g,n}$ has $\lambda_{1}(X)\leqslant\frac{1}{4}-\frac{\left(2\alpha+1\right)^{2}}{16}-\varepsilon$,
then since $\alpha\in[0,\frac{1}{2})$ this guarantees that $\lambda_{1}(X)\leqslant\frac{3}{16}$
and we can apply Theorem \ref{Theorem: Main analytic Theorem } to
see 
\[
C(\varepsilon)\log\left(g\right)g^{4(1-\varepsilon)\sqrt{\frac{1}{4}-\lambda_{1}(X)}}\leqslant\sum_{\gamma\in\mathcal{P}(X)}\sum_{k=1}^{\infty}H_{X,k}(\gamma)-\hat{f}_{T}\left(\frac{i}{2}\right)+O\left(ng\right).
\]
But since $\ep<\frac{1}{2}-\alpha$,
\[
\sqrt{\frac{1}{4}-\lambda_{1}(X)}\geqslant\frac{2\alpha+1}{4}+\ep,
\]
 and we deduce that 
\begin{align*}
C(\varepsilon)\log\left(g\right)g^{4(1-\varepsilon)\sqrt{\frac{1}{4}-\lambda_{1}(X)}} & \geqslant C(\varepsilon)\log\left(g\right)g^{(1-\varepsilon)\left((2\alpha+1)+4\ep\right)}\gg_{\epsilon}g^{2\alpha+1+2\varepsilon-4\varepsilon^{2}}>n^{2}g^{1+\ep},
\end{align*}
for sufficiently large $g$. On the last line we used that $\ep<\frac{1}{4}$
and that $n=O\left(g^{\alpha}\right)$. We deduce that 
\[
\sum_{\gamma\in\mathcal{P}(X)}\sum_{k=1}^{\infty}H_{X,k}(\gamma)-\hat{f}_{T}\left(\frac{i}{2}\right)>n^{2}g^{1+\ep},
\]
for sufficiently large $g$. This tells us that for $g$ sufficiently
large,
\begin{align*}
\mathbb{P}_{WP}^{g,n}\left[\tilde{\lambda}_{1}\left(X\right)\leqslant\frac{1}{4}-\frac{\left(2\alpha+1\right)^{2}}{16}-\varepsilon\right] & \leqslant\mathbb{P}_{WP}^{g,n}\left[\sum_{\gamma\in\mathcal{P}(X)}\sum_{k=1}^{\infty}H_{X,k}(\gamma)-\hat{f}_{T}\left(\frac{i}{2}\right)+\nu ng>n^{2}g^{1+\ep}\right]\\
 & \ll_{\ep}\left(1+\frac{\log\left(g\right)^{6}}{n}+\left(\log g\right)^{\beta+1}\right)g^{-\frac{\ep}{2}}\to0,
\end{align*}
as $g\to\infty$.
\end{proof}

\appendix

\section{Volume estimates\label{sec:Volume-estimates}}

The purpose of this appendix is to prove the necessary Weil-Petersson
volume estimates used in the proof of Theorem \ref{Main Geometric Theorem}.
Similar estimates can be found in e.g. \cite{Mi13,MP19,NWX20,GMST21,LW21}.

We need the following lemma in the proof of Lemma \ref{Lema contribution of simple and separating}
and Lemma \ref{contribution of cimple and non-separating}. 
\begin{lem}
\label{lem:Sinh bound-1}Let $x_{1},\dots,x_{n}\geqslant0$. For $g,n\geqslant0$,
$2g-2+n>0$ we have 
\[
\frac{V_{g,n}\left(x_{1},...,x_{n}\right)}{V_{g,n}}\leqslant\prod_{i=1}^{n}\frac{\sinh\left(\frac{x_{i}}{2}\right)}{\left(\frac{x_{i}}{2}\right)},
\]
and 
\[
\frac{V_{g,n}\left(\underline{0}_{n-2},x_{1},x_{2}\right)}{V_{g,n}}=\frac{4\sinh\left(\frac{x_{1}}{2}\right)\cdot\sinh\left(\frac{x_{2}}{2}\right)}{x_{1}\cdot x_{2}}\left(1+O\left(\frac{n\left(x_{1}^{2}+x_{2}^{2}\right)}{g}\right)\right),
\]
as $g\to\infty$, where the implied constant is independent of $n$.
\end{lem}

\begin{rem}
Lemma \ref{lem:Sinh bound-1} is due to \cite[Proposition 3.1]{MP19}
and \cite[Lemma 20]{NWX20}. The proof of the second statement is
identical to the proof of \cite[Lemma 20]{NWX20}, if one uses \cite[Theorem A.1]{LW21}
in place of \cite[Page 286]{Mi13}.
\end{rem}

We require estimates for $V_{g,n}$ where the number of cusps $n$
is allowed to grow with the genus $g$. The starting point is the
following theorem of Mirzakhani and Zograf.
\begin{thm}[{\cite[Theorem 1.8]{MZ15}}]
\label{Variable n asymptotic}There exists a constant $B>0$ such
that if $n=o\left(g^{\frac{1}{2}}\right)$, we have
\[
V_{g,n(g)}=\frac{B}{\sqrt{g}}\left(2g-3+n\left(g\right)\right)!\left(4\pi^{2}\right)^{2g-3+n\left(g\right)}\left(1+O\left(\frac{1+n\left(g\right)^{2}}{g}\right)\right),
\]
as $g\to\infty.$
\end{thm}

In order to control the contribution of simple separating geodesics,
in Lemma \ref{Lema contribution of simple and separating} we need
the following lemma.
\begin{lem}
\label{Summing over}If $n=o\left(g^{\frac{1}{2}}\right)$, then
\[
\sum_{\substack{0\leqslant i\leqslant g,\text{}0\leqslant j\leqslant n\\
2\leqslant2i+j\leqslant2g+n-2
}
}{n \choose j}\cdot\frac{V_{i,j+1}V_{g-i,n-j+1}}{V_{g,n}}\ll\frac{1+n^{2}}{g}.
\]
\end{lem}

The case that $n$ is fixed is treated in \cite[Lemma 3.3]{Mi13}.
The fact that the number of cusps is growing with genus and the presence
of the multiplicity ${n \choose j}$ presents the new difficulty here.

In the following, we shall frequently apply Stirling's approximation
which tells us that there exist constants $1<c_{1}<c_{2}<2$ with
\begin{equation}
c_{1}\cdot\sqrt{2\pi w}\left(\frac{w}{e}\right)^{w}<w!<c_{2}\cdot\sqrt{2\pi w}\left(\frac{w}{e}\right)^{w},\label{Sterling's approximation}
\end{equation}
for all $w\geqslant1$. 
\begin{proof}[Proof of Lemma \ref{Summing over}]
By Theorem \ref{Variable n asymptotic}, since $n=o\left(\sqrt{g}\right)$,
we have 
\begin{equation}
V_{g,n(g)}=\frac{B}{\sqrt{g}}\left(2g-3+n\right)!\left(4\pi^{2}\right)^{2g-3+n}\left(1+O\left(\frac{1+n^{2}}{g}\right)\right).\label{eq:ng}
\end{equation}
By \cite[Lemma 3.2, part 3]{Mi13} we have that for $a,b\geqslant0$,
$2a+b\geqslant1$,
\begin{equation}
V_{a,b+4}\leqslant V_{a+1,b+2}.\label{eq:Reducing cusps-1}
\end{equation}
Applying (\ref{eq:Reducing cusps-1}) iteratively, for $j\geqslant1$,
\[
V_{i,j+1}\leqslant V_{i+\lfloor\frac{j-1}{2}\rfloor,j+1-2\lfloor\frac{j-1}{2}\rfloor}.
\]
We can then apply Theorem \ref{Variable n asymptotic} to see that
\begin{align}
V_{i,j+1}V_{g-i,n-j+1} & \ll\left(4\pi^{2}\right)^{2g+n}\frac{\left(2i+j-2\right)!}{\sqrt{i+\max\left\{ \lfloor\frac{j-1}{2}\rfloor,0\right\} }}\cdot\frac{\left(2g-2i+n-j-2\right)!}{\sqrt{g-i+\max\left\{ \lfloor\frac{n-j-1}{2}\rfloor,0\right\} }}.\label{inj-1}
\end{align}
We also observe that 
\begin{equation}
\frac{\sqrt{g}}{\sqrt{g-i+\max\left\{ \lfloor\frac{n-j-1}{2}\rfloor,0\right\} }\cdot\sqrt{i+\max\left\{ \lfloor\frac{j-1}{2}\rfloor,0\right\} }}\ll1.\label{eq:relgenusbound}
\end{equation}
Then applying (\ref{eq:ng}), $\eqref{inj-1}$ and (\ref{eq:relgenusbound}),

\begin{align*}
 & \sum_{\substack{0\leqslant i\leqslant g,\text{}0\leqslant j\leqslant n\\
2\leqslant2i+j\leqslant2g+n-2
}
}\frac{n!}{j!\left(n-j\right)!}\cdot\frac{V_{i,j+1}V_{g-i,n-j+1}}{V_{g,n}}\\
\ll & \sum_{\substack{0\leqslant i\leqslant g,\text{}0\leqslant j\leqslant n\\
2\leqslant2i+j\leqslant2g+n-2
}
}\frac{n!}{j!\left(n-j\right)!}\frac{\left(2i+j-2\right)!\left(2g-2i+n-j-2\right)!}{\left(2g+n-3\right)!}.
\end{align*}
If $i=0$ then $j\geqslant2$ and we have 
\begin{align*}
\sum_{j=2}^{n}\frac{n!}{j!\left(n-j\right)!}\cdot\frac{\left(j-2\right)!\left(2g+n-j-2\right)!}{\left(2g+n-3\right)!}= & \sum_{j=2}^{n-4}\frac{n!}{j\left(j-1\right)\left(n-j\right)!}\cdot\frac{\left(2g+n-j-2\right)!}{\left(2g+n-3\right)!}\\
\ll & \frac{n^{2}}{g}+\sum_{j=3}^{n-4}\frac{n^{j}}{g^{j-1}}\ll\frac{n^{2}}{g},
\end{align*}
since $n=o\left(\sqrt{g}\right)$. By symmetry, the same calculation
holds for the case that $i=g$. Similarly, if $i=1$ then $j\geqslant0$
and we calculate 
\begin{align*}
\sum_{j=0}^{n}\frac{n!}{j!\left(n-j\right)!}\cdot\frac{j!\left(2g+n-j-4\right)!}{\left(2g+n-3\right)!} & \ll\sum_{j=0}^{n}\frac{n^{j}}{g^{j+1}}\ll\frac{1}{g}.
\end{align*}
The same calculation holds in the case that $i=g-1$ by symmetry.
If $2\leqslant i\leqslant g-2$ then we claim that
\begin{align}
\frac{n!}{j!\left(n-j\right)!}\frac{\left(2i+j-2\right)!\left(2g-2i+n-j-2\right)!}{\left(2g+n-3\right)!} & \ll g^{-3}.\label{eq:claim first proof}
\end{align}
It is a straightforward calculation to check that (\ref{eq:claim first proof})
holds in the case that $i=2$, $j=0$ and $i=2,j=1$. Now let $L=2i+j$.
Then if $6\leqslant L\le n$, 
\begin{align*}
\frac{n!}{j!\left(n-j\right)!}\frac{\left(2i+j-2\right)!\left(2g-2i+n-j-2\right)!}{\left(2g+n-3\right)!} & \ll\frac{L!\cdot n^{L}}{g^{L}}\ll\sqrt{L}\left(\frac{Ln}{ge}\right)^{L},
\end{align*}
by Stirling's approximation. If $L=6$ then
\[
\sqrt{L}\left(\frac{Ln}{ge}\right)^{L}\ll\left(\frac{n}{g}\right)^{6}\ll g^{-3}.
\]
If $6<L\leqslant n-1$ then 
\begin{align*}
\sqrt{L}\left(\frac{Ln}{ge}\right)^{L}=\sqrt{L}\left(\frac{Ln}{ge}\right)^{L}\left(\frac{6n}{ge}\right)^{6}\cdot\left(\frac{eg}{6n}\right)^{6} & \ll\sqrt{L}g^{-3}\left(\frac{L}{6}\right)^{6}\left(\frac{Ln}{ge}\right)^{L-6}\\
 & \leqslant L^{\frac{13}{2}}e^{6-L}g^{-3}\cdot\frac{n^{2}}{g}\ll g^{-3}.
\end{align*}
If $n\leqslant L\leqslant\frac{1}{2}\left(2g+n-2\right)$, then since
\[
{n \choose i}\leqslant2^{n},
\]
we have 
\begin{align*}
\frac{n!}{j!\left(n-j\right)!}\frac{\left(2i+j-2\right)!\left(2g-2i+n-j-2\right)!}{\left(2g+n-3\right)!} & \ll2^{n}\frac{L!\left(2g+n-2-L\right)!}{\left(2g+n-3\right)!}\\
 & \leqslant\frac{2^{n}n!\left(2g-2-n\right)!}{\left(2g+n-3\right)!}\ll\left(\frac{2n}{g}\right)^{n}\ll g^{-3}.
\end{align*}
By symmetry, the case that $2i+j\geqslant\frac{1}{2}\left(2g+n-2\right)$
is treated analogously. This establishes the claim (\ref{eq:claim first proof}).
We can now use the rough bound 
\[
\#\{(i,j)\in\mathbb{Z}_{\geqslant0}\mid2\leqslant i\leqslant g-2,\text{}0\leqslant j\leqslant n,\text{}2\leqslant2i+j\leqslant2g+n-2\}\ll ng,
\]
to deduce that
\[
\sum_{\substack{2\leqslant i\leqslant g-2,\text{}0\leqslant j\leqslant n\\
2\leqslant2i+j\leqslant2g+n-2
}
}\frac{n!}{j!\left(n-j\right)!}\frac{\left(2i+j-2\right)!\left(2g-2i+n-j-2\right)!}{\left(2g+n-3\right)!}\ll\frac{n}{g^{2}},
\]
and the result follows.
\end{proof}
In order to deal with the contribution of non-simple geodesics, we
needed the following Lemma. 
\begin{lem}
\label{WpVOLest}Let $n=o\left(\sqrt{g}\right)$ and let $g_{0},a_{0},n_{0}$
and $k$ be given with $m=2g_{0}+a_{0}+k-2\leqslant3\log g-2$. For
$1\leqslant q\leqslant k-2n_{0}$, 
\[
\sum_{\left\{ \left(g_{j},a_{j},n_{j}\right)\right\} _{i=1}^{q}\in\mathcal{A}}\frac{n!}{a_{0}!\cdots a_{q}!}.\frac{V_{g_{0},n_{0}+a_{0}}\cdots V_{g_{q},n_{q}+a_{q}}}{V_{g,n}}\ll\left(2g_{0}+k+a_{0}-3\right)!\frac{n^{a_{0}}}{g^{m}},
\]
where the summation is taken over the set $\mathcal{A}$ of all ``admissible
triples'' $\left\{ \left(g_{1},a_{1},n_{1}\right),\dots,\left(g_{q},a_{q},n_{q}\right)\right\} $
where $g_{j},a_{j}\geqslant0$, $n_{j}\geqslant1$ and $2g_{j}+a_{j}+n_{j}\geqslant3$
such that
\begin{lyxlist}{00.00.0000}
\item [{i)}] $\sum_{i=1}^{q}\left(2g_{i}-2+n_{i}+a_{i}\right)=2g-2+n-m$,
\item [{ii)}] $\sum_{i=1}^{q}n_{i}=k-2n_{0}$, 
\item [{iii)}] $\sum_{j=1}^{q}a_{j}=n-a_{0}$. 
\end{lyxlist}
\end{lem}

This is similar to estimates proved in \cite{WX21} but here we need
the number of cusps to grow with genus and we have the extra multiplicity
\[
\frac{n!}{a_{0}!\cdot\cdots\cdot a_{q}!}.
\]
We take a similar approach as in the proof of Lemma \ref{Summing over}.
Lemma \ref{WpVOLest} relies on a lot of computations which, for the
sake of readability, are done separately in Lemma \ref{computations}.
\begin{proof}[Proof of Lemma \ref{WpVOLest} given Lemma \ref{computations}]
 By \cite[Lemma 3.2, part 3]{Mi13} we see that for each $a_{i}+n_{i}\geqslant2$,
we have 
\[
V_{g_{i},a_{i}+n_{i}}\leqslant V_{g_{i}+\lfloor\frac{a_{i}+n_{i}-2}{2}\rfloor,a_{i}+n_{i}-2\lfloor\frac{a_{i}+n_{i}-2}{2}\rfloor}.
\]
 This allows us to apply Theorem \ref{Variable n asymptotic} which
tells us that there exists $C_{1}>0$ with
\begin{align}
V_{g_{1},n_{1}+a_{1}}\cdots V_{g_{q},n_{q}+a_{q}} & \leqslant C_{1}^{q}\prod_{j=1}^{q}\frac{\left(4\pi^{2}\right)^{2g_{j}+a_{j}+n_{j}-3}\left(2g_{j}+a_{j}+n_{j}-3\right)!}{\sqrt{g_{j}+\max\left\{ \lfloor\frac{a_{j}+n_{j}-2}{2}\rfloor,0\right\} }},\label{eq:WPvolreduced}
\end{align}
where since $V_{0,3}=1$ we interpret the product in (\ref{eq:WPvolreduced})
as only over triples with $g_{j}+\max\left\{ \lfloor\frac{a_{j}+n_{j}-2}{2}\rfloor,0\right\} >0$.
We also see by Theorem \ref{Variable n asymptotic} that 
\begin{equation}
V_{g_{0},a_{0}+k}\leqslant C_{1}\left(4\pi^{2}\right)^{2g_{0}+a_{0}+k-3}\left(2g_{0}+a_{0}+k-3\right)!,\label{gnought}
\end{equation}
and 
\begin{equation}
V_{g,n}=\frac{B}{\sqrt{g}}\left(2g-3+n\left(g\right)\right)!\left(4\pi^{2}\right)^{2g-3+n\left(g\right)}\left(1+O\left(\frac{1+n\left(g\right)^{2}}{g}\right)\right).\label{Vg-asymp}
\end{equation}
We introduce the notation $\text{\ensuremath{\overline{a_{j}+n_{j}}}}\eqdf\max\left\{ \lfloor\frac{a_{j}+n_{j}-2}{2}\rfloor,0\right\} $.
By applying (\ref{eq:WPvolreduced}), (\ref{gnought}) and (\ref{Vg-asymp})
and noting that $n_{i}!\geqslant1$ for each $i$, we calculate that 

\begin{align*}
 & \sum_{\mathcal{A}}\frac{n!}{a_{0}!\cdot\cdots\cdot a_{q}!}.\frac{V_{g_{0},n_{0}+k}\cdot V_{g_{1},n_{1}+a_{1}}\cdot\cdots\cdot V_{g_{q},n_{q}+a_{q}}}{V_{g,n}\cdot n_{0}!n_{1}!\cdots n_{q}!}\\
\ll & \left(2g_{0}+k+a_{0}-3\right)!\sum_{\mathcal{A}}\frac{C_{1}^{q}\sqrt{g}}{\prod_{j=1}^{q}\sqrt{g_{j}+\ensuremath{\overline{a_{j}+n_{j}}}}}\frac{n!}{\prod_{j=0}^{q}a_{j}!}\frac{\prod_{j=1}^{q}\left(2g_{j}+a_{j}+n_{j}-3\right)!}{\left(2g+n-3\right)!}.
\end{align*}
The result then follows from the fact that 
\[
\sum_{\mathcal{A}}\frac{C_{1}^{q}\sqrt{g}}{\prod_{j=1}^{q}\sqrt{g_{j}+\ensuremath{\overline{a_{j}+n_{j}}}}}\frac{n!}{\prod_{j=0}^{q}a_{j}!}\frac{\prod_{j=1}^{q}\left(2g_{j}+a_{j}+n_{j}-3\right)!}{\left(2g+n-3\right)!}\ll\frac{n^{a_{0}}}{g^{m}},
\]
which is proved in Lemma \ref{computations}.
\end{proof}
We now need to prove Lemma \ref{computations}, which is purely computational.
\begin{lem}
\label{computations}Let $n=o\left(\sqrt{g}\right)$ and let $g_{0},a_{0},n_{0}$
and $k$ be given with $m=2g_{0}+a_{0}+k-2\leqslant3\log g-2$ and
$1\leqslant q\leqslant k-2n_{0}$. With $\mathcal{A}$ as in Lemma
\ref{WpVOLest}, we have
\begin{equation}
\sum_{\mathcal{A}}\frac{C_{1}^{q}\sqrt{g}}{\prod_{j=1}^{q}\sqrt{g_{j}+\ensuremath{\overline{a_{j}+n_{j}}}}}\frac{n!}{\prod_{j=0}^{q}a_{j}!}\frac{\prod_{j=1}^{q}\left(2g_{j}+a_{j}+n_{j}-3\right)!}{\left(2g+n-3\right)!}\ll\frac{n^{a_{0}}}{g^{m}}.\label{Grimness}
\end{equation}
\end{lem}

In the proof of Lemma \ref{computations}, we will frequently apply
the following observation: if $x_{i}\geqslant0$ with $\sum_{i=1}^{s}x_{i}=A$,
then 
\begin{equation}
\prod_{i=1}^{s}x_{i}!\leqslant A!,\label{eq:Elementary}
\end{equation}
which can be seen by the fact that the multinomial coefficient ${A \choose x_{1},\dots,x_{s}}$
is bounded below by $1$. 
\begin{proof}
We first note that $q\leqslant3\log g$. For $\left\{ \left(g_{1},a_{1},n_{1}\right),\dots,\left(g_{q},a_{q},n_{q}\right)\right\} \in\mathcal{A}$,
we claim that if $\max_{1\leqslant i\leqslant q}\left(2g_{i}+a_{i}+n_{i}-3\right)\leqslant2g+n-3-m-8q$
then 
\begin{equation}
\frac{C_{1}^{q}\sqrt{g}}{\prod_{j=1}^{q}\sqrt{g_{j}+\ensuremath{\overline{a_{j}+n_{j}}}}}\frac{n!}{\prod_{j=0}^{q}a_{j}!}\frac{\prod_{j=1}^{q}\left(2g_{j}+a_{j}+n_{j}-3\right)!}{\left(2g+n-3\right)!}\ll g^{-\frac{7}{2}q}.\label{eq:Claim}
\end{equation}
This estimate is analogous to (\ref{eq:claim first proof}). Once
we have established (\ref{eq:Claim}) we shall apply a rough counting
argument to bound the contribution of such terms to the sum (\ref{Grimness}). 

Let $\max_{1\leqslant i\leqslant q}\left(2g_{i}+a_{i}+n_{i}-3\right)=2g+n-3-m-L$.
First we treat the case that $L\geqslant\frac{1}{2}\left(2g+n-m-3\right)$.
We apply Stirling's approximation (\ref{Sterling's approximation})
to see that 
\begin{align}
\frac{\left(2g_{i}+n_{i}+a_{i}-3\right)!}{\sqrt{g_{j}+\ensuremath{\overline{a_{j}+n_{j}}}}} & <c_{2}\frac{\sqrt{2\pi\left(2g_{i}+n_{i}+a_{i}-3\right)}}{\sqrt{g_{j}+\ensuremath{\overline{a_{j}+n_{j}}}}}\cdot\left(\frac{2g_{i}+a_{i}+n_{i}-3}{e}\right)^{2g_{i}+a_{i}+n_{i}-3}\nonumber \\
 & <4\sqrt{\pi}\cdot\left(\frac{2g_{i}+a_{i}+n_{i}-3}{e}\right)^{2g_{i}+a_{i}+n_{i}-3}.\label{stirling 1}
\end{align}
 Applying Stirling's approximation again, we see that 
\begin{align}
\frac{\sqrt{g}}{\left(2g+n-3\right)!} & >\frac{1}{c_{2}}\frac{\sqrt{g}}{\sqrt{2\pi\left(2g+n-3\right)}}\cdot\left(\frac{e}{2g+n-3}\right)^{2g+n-3}\nonumber \\
 & \gg\left(\frac{e}{2g+n-3}\right)^{2g+n-3}.\label{stirling 2}
\end{align}
We also note that 
\begin{align}
\frac{n!}{a_{0}!\cdots a_{q}!} & \leqslant q^{n},\label{eq:Bound on n contribution}
\end{align}
by the multinomial theorem. By (\ref{stirling 1}), (\ref{stirling 2})
and (\ref{eq:Bound on n contribution}),
\begin{align}
 & \frac{C_{1}^{q}\sqrt{g}}{\prod_{j=1}^{q}\sqrt{g_{j}+\ensuremath{\overline{a_{j}+n_{j}}}}}\frac{n!}{\prod_{j=0}^{q}a_{j}!}\frac{\prod_{j=1}^{q}\left(2g_{j}+a_{j}+n_{j}-3\right)!}{\left(2g+n-3\right)!}\nonumber \\
\ll & q^{n}\cdot\left(4C_{1}\sqrt{\pi}\right)^{q}\frac{\prod_{j=1}^{q}\left(2g_{j}+a_{j}+n_{j}-3\right)^{\left(2g_{j}+a_{j}+n_{j}-3\right)}}{\left(2g+n-3\right)^{\left(2g+n-3\right)}}.\label{stirling 3}
\end{align}
We now bound the expression in (\ref{stirling 3}). Given $s$ integers
$x_{i}>0$, Jensen's inequality for concave functions applied to the
function $\log x$ tells us that
\[
\frac{\sum_{i=1}^{s}x_{i}\log x_{i}}{\sum_{i=1}^{s}x_{i}}\leqslant\log\left(\frac{\sum_{i=1}^{s}x_{i}^{2}}{\sum_{i=1}^{s}x_{i}}\right).
\]
If $\sum_{i=1}^{s}x_{i}=A$ and $\max_{1\leqslant i\leqslant s}x_{i}=B$,
then 
\begin{align*}
\sum_{i=1}^{s}x_{i}\log x_{i} & \leqslant A\log\left(\frac{\sum_{i=1}^{s}x_{i}^{2}}{\sum_{i=1}^{s}x_{i}}\right)\\
 & \leqslant A\log B,
\end{align*}
and by exponentiating, we conclude that
\begin{equation}
\prod_{i=1}^{s}x_{i}^{x_{i}}\leqslant B^{A}.\label{eq:exponentbound}
\end{equation}
Note that (\ref{eq:exponentbound}) also holds if instead we just
require $x_{i}\geqslant0$ since we can apply Jensen's inequality
with only the non-zero terms. Recall that $\max_{1\leqslant i\leqslant q}\left(2g_{i}+a_{i}+n_{i}-3\right)=2g+n-3-m-L$
for $L\geqslant\frac{1}{2}\left(2g+n-m-3\right)$. Since $\sum_{i=1}^{q}\left(2g_{i}-2+n_{i}+a_{i}\right)=2g-2+n-m$,
then in particular, $\sum_{i=1}^{q}\left(2g_{i}-3+n_{i}+a_{i}\right)\leqslant2g+n-m-3$
and we can apply (\ref{eq:exponentbound}) to (\ref{stirling 3})
to calculate that
\begin{align*}
 & q^{n}\cdot\left(4C_{1}\sqrt{\pi}\right)^{q}\frac{\prod_{j=1}^{q}\left(2g_{j}+a_{j}+n_{j}-3\right)^{\left(2g_{j}+a_{j}+n_{j}-3\right)}}{\left(2g+n-3\right)^{\left(2g+n-3\right)}}\\
\ll & q^{n}\cdot\left(4C_{1}\sqrt{\pi}\right)^{q}\cdot\frac{\left(2g+n-3-m-L\right)^{\left(2g+n-3\right)}}{\left(2g+n-3\right)^{\left(2g+n-3\right)}}\\
\leqslant & q^{n}\cdot2^{3q}C_{1}^{q}\left(\frac{1}{2}\right)^{2g+n-3}\leqslant q^{n}\left(\frac{1}{2}\right)^{2g+n-3-3q-q\log_{2}C_{1}}.
\end{align*}
Since $q\leqslant3\log g$ and $n=o\left(\sqrt{g}\right)$,
\begin{align*}
q^{n}\left(\frac{1}{2}\right)^{2g+n-3-3q-q\log_{2}C_{1}} & \ll\left(\frac{1}{2}\right)^{g}=g^{-\frac{g}{\log_{2}g}}\ll g^{-\frac{7}{2}q}.
\end{align*}
This justifies the claim in the case that $L\geqslant\frac{1}{2}\left(2g+n-m-3\right)$.

In order to treat the remaining cases, we first make the following
observation. Recalling that $\sum_{j=1}^{q}\left(2g_{j}+a_{j}+n_{j}-3\right)=2g+n-m-3-\left(q-1\right)$
and that $\text{\ensuremath{\overline{a_{j}+n_{j}}}}\eqdf\max\left\{ \lfloor\frac{a_{j}+n_{j}-2}{2}\rfloor,0\right\} $,
we see that
\begin{align*}
\sum_{j=1}^{q}\left(g_{j}+\overline{a_{j}+n_{j}}\right) & \geqslant\frac{1}{2}\sum_{j=1}^{q}\left(2g_{j}+a_{j}+n_{j}-3\right)\geqslant\frac{2g+n-m-3-\left(q-1\right)}{2}.
\end{align*}
For any $q$ positive integers $x_{i}$, we have 
\[
\prod_{i=1}^{q}x_{i}\geqslant\sum_{i=1}^{q}x_{i}-\left(q-1\right).
\]
Then
\begin{align*}
\prod_{j=1}^{q}\left(g_{j}+\overline{a_{j}+n_{j}}\right) & \geqslant\frac{2g+n-m-3-\left(q-1\right)}{2}-q-1\gg g,
\end{align*}
since $n=o\left(\sqrt{g}\right)$ and $q,m=O\left(\log g\right)$.
We see that
\begin{align*}
\frac{\sqrt{g}}{\prod_{j=1}^{q}\sqrt{g_{j}+\ensuremath{\overline{a_{j}+n_{j}}}}} & \ll1,
\end{align*}
and therefore 
\begin{align}
 & \frac{C_{1}^{q}\sqrt{g}}{\prod_{j=1}^{q}\sqrt{g_{j}+\overline{a_{j}+n_{j}}}}\frac{n!}{\prod_{j=0}^{q}a_{j}!}\frac{\prod_{j=1}^{q}\left(2g_{j}+a_{j}+n_{j}-3\right)!}{\left(2g+n-3\right)!}\nonumber \\
\ll & \frac{C_{1}^{q}n!}{\prod_{j=0}^{q}a_{j}!}\frac{\prod_{j=1}^{q}\left(2g_{j}+a_{j}+n_{j}-3\right)!}{\left(2g+n-3\right)!}.\label{eq:Boundings}
\end{align}
The expression in (\ref{eq:Boundings}) will be easier to work with
for the remaining cases. Recalling that $\max_{1\leqslant i\leqslant q}\left(2g_{i}+a_{i}+n_{i}-3\right)=2g+n-3-m-L$,
we now treat the case that $8q\leqslant L\leqslant n-a_{0}$. Since
$\max_{1\leqslant i\leqslant q}\left(2g_{i}+a_{i}+n_{i}-3\right)=2g+n-3-m-L$,
this forces $\max_{1\leqslant i\leqslant q}a_{i}\geqslant n-a_{0}-L$.
Indeed if $\max_{1\leqslant i\leqslant q}a_{i}<n-a_{0}-L$ we would
have that 
\[
\max_{1\leqslant i\leqslant q}\left(2g_{i}+n_{i}\right)>2g-2g_{0}-n_{0},
\]
which is not possible. Since there is an $1\leqslant i\leqslant q$
such that $2g_{i}+a_{i}+n_{i}-3=2g+n-3-m-L$ and we have $\sum_{j=1,j\neq q}^{q}\left(2g_{i}+a_{i}+n_{i}-3\right)=L-\left(q-1\right)\leqslant L$,
we apply (\ref{eq:Elementary}) to see that
\begin{align}
\prod_{j=1}^{q}\left(2g_{j}+a_{j}+n_{j}-3\right)! & =\left(2g+n-3-m-L\right)!\prod_{j=1,j\neq i}^{q}\left(2g_{j}+a_{j}+n_{j}-3\right)!\nonumber \\
 & \leqslant L!\left(2g+n-3-m-L\right)!.\label{eq:bounded v}
\end{align}
We then use the rough bound
\begin{equation}
\frac{n!}{\prod_{j=0}^{q}a_{j}!}\leqslant\frac{n!}{\left(\max_{1\leqslant i\leqslant q}a_{i}\right)!}\leqslant\frac{n!}{\left(n-a_{0}-L\right)!}\ll n^{a_{0}+L},\label{Other bound with n}
\end{equation}
together with (\ref{eq:bounded v}), to see that
\begin{align}
\frac{n!}{\prod_{j=0}^{q}a_{j}!}\frac{\prod_{j=1}^{q}\left(2g_{j}+a_{j}+n_{j}-3\right)!}{\left(2g+n-3\right)!} & \ll\frac{n^{a_{0}+L}L!\left(2g+n-3-m-L\right)!}{\left(2g+n-3\right)!}\ll\frac{n^{a_{0}+L}}{g^{m+L}}L!.\label{eq:upperboundL}
\end{align}
By applying Stirling's approximation (\ref{Sterling's approximation}),
\[
\frac{n^{a_{0}+L}}{g^{m+L}}L!\ll\sqrt{L}\left(\frac{n\cdot L}{e\cdot g}\right)^{L}\cdot\frac{n^{a_{0}}}{g^{m}}.
\]
If $L=8q$ then since $n=o\left(\sqrt{g}\right)$ and $q\leqslant3\log g$,
\begin{align}
C_{1}^{q}\sqrt{L}\left(\frac{n\cdot L}{e\cdot g}\right)^{L} & \ll C_{1}^{q}g^{-4q}\left(8q\right)^{8q+\frac{1}{2}}\ll g^{\frac{-7q}{2}}.\label{Equation45}
\end{align}
Now if $8q<L\leqslant n-a_{0}$,
\begin{align*}
C_{1}^{q}\frac{n^{a_{0}}}{g^{m}}.\sqrt{L}\left(\frac{n\cdot L}{e\cdot g}\right)^{L} & \ll C_{1}^{q}\frac{\sqrt{L}}{\sqrt{g}}\left(\frac{n\cdot L}{e\cdot g}\right)^{L}\cdot\left(\frac{n\cdot8q}{e\cdot g}\right)^{8q}\cdot\left(\frac{e\cdot g}{n\cdot8q}\right)^{8q}\\
 & \ll g^{\frac{-7q}{2}}\cdot\left(\frac{L}{8q}\right)^{8q}\cdot\left(\frac{n\cdot L}{e\cdot g}\right)^{L-8q}\\
 & \leqslant g^{\frac{-7q}{2}}\cdot e^{L-8q}\cdot\left(\frac{n\cdot L}{e\cdot g}\right)^{L-8q}\ll g^{\frac{-7q}{2}},
\end{align*}
which justifies the claim (\ref{eq:Claim}) in the case that $8q\leqslant L\leqslant n-a_{0}$.
Finally we treat the case that $8q<n-a_{0}<L\leqslant\frac{2g+n-3-m}{2}$.
We calculate, with (\ref{eq:bounded v}) and (\ref{eq:Bound on n contribution}),
that
\begin{align}
\frac{C_{1}^{q}n!}{\prod_{j=0}^{q}a_{j}!}\frac{\prod_{j=1}^{q}\left(2g_{j}+a_{j}+n_{j}-3\right)!}{\left(2g+n-3\right)!} & \ll\frac{C_{1}^{q}\cdot q^{n}\cdot L!\left(2g+n-m-3-L\right)!}{\left(2g+n-3\right)!}\nonumber \\
 & \ll\frac{C_{1}^{q}\cdot q^{n}\left(n-a_{0}\right)!\left(2g+a_{0}-m-3\right)!}{\left(2g+n-m\right)!}\nonumber \\
 & \ll\frac{g^{3\log C_{1}}\left(3\log g\right)^{n+1}n^{n}}{\left(2g\right)^{n}}\ll g^{-\frac{7}{2}q},\label{eq:Claim part 2}
\end{align}
which justifies the claim (\ref{eq:Claim}) for $8q<n-a_{0}<L\leqslant\frac{2g+n-3-m}{2}$.
Note that in the case that $n\leqslant8q-n_{0}$ we can simply apply
the argument in (\ref{eq:Claim part 2}) with $L\geqslant8q$. The
claim (\ref{eq:Claim}) is now proved.

Now we have established (\ref{eq:Claim}), we apply the very rough
bound for the size of the set $\mathcal{A}$,
\[
\left|\mathcal{A}\right|\ll g^{3q},
\]
together with (\ref{eq:Claim}) to calculate
\begin{align}
 & \sum_{\substack{\left\{ \left(g_{i},a_{i},n_{i}\right)\right\} _{i=1}^{q}\in\mathcal{A}\\
\max_{1\leqslant i\leqslant q}\left(2g_{i}+a_{i}+n_{i}-3\right)\leqslant2g+n-2-m-8q
}
}\frac{C_{1}^{q}\sqrt{g}}{\prod_{j=1}^{q}\sqrt{g_{j}+\ensuremath{\overline{a_{j}+n_{j}}}}}\frac{n!\prod_{j=1}^{q}\left(2g_{j}+a_{j}+n_{j}-3\right)!}{\prod_{j=0}^{q}a_{j}!\left(2g+n-3\right)!}\nonumber \\
\ll & \frac{n^{a_{0}}}{g^{m}}\sum_{\substack{\left\{ \left(g_{i},a_{i},n_{i}\right)\right\} _{i=1}^{q}\in\mathcal{A}\\
\max_{1\leqslant i\leqslant q}\left(2g_{i}+a_{i}+n_{i}-3\right)\leqslant2g+n-2-m-8q
}
}g^{-\frac{7}{2}q}\ll\left|\mathcal{A}\right|\cdot\frac{n^{a_{0}}}{g^{m}}\cdot g^{-\frac{7}{2}q}\ll\frac{n^{a_{0}}}{g^{m}}\cdot g^{-\frac{q}{2}}.\label{eq:Summationpart1}
\end{align}
We now consider the sum
\begin{equation}
\sum_{\substack{\left\{ \left(g_{i},a_{i},n_{i}\right)\right\} _{i=1}^{q}\in\mathcal{A}\\
\max_{1\leqslant i\leqslant q}\left(2g_{i}+a_{i}+n_{i}-3\right)>2g+n-3-m-8q
}
}\frac{C_{1}^{q}\sqrt{g}}{\prod_{j=1}^{q}\sqrt{g_{j}+\ensuremath{\overline{a_{j}+n_{j}}}}}\frac{n!\prod_{j=1}^{q}\left(2g_{j}+a_{j}+n_{j}-3\right)!}{\prod_{j=0}^{q}a_{j}!\left(2g+n-3\right)!}.\label{eq:less than q}
\end{equation}
 Let $\max_{1\leqslant i\leqslant q}\left(2g_{i}+a_{i}+n_{i}-3\right)=2g+n-3-m-L$.
Since $2g_{j}+a_{j}+n_{j}-3\geqslant0$ and $\sum_{j=1}^{q}2g_{j}+a_{j}+n_{j}-3=2g+n-m-3-\left(q-1\right)$,
we see that $L\geqslant q-1$. By the same arguments as in (\ref{eq:upperboundL})
and (\ref{Equation45}), if $q-1\leqslant L\leqslant8q\leqslant24\log g$
then 
\begin{equation}
\frac{n!}{\prod_{j=0}^{q}a_{j}!}\frac{\prod_{j=1}^{q}\left(2g_{j}+a_{j}+n_{j}-3\right)!}{\left(2g+n-3\right)!}\ll\frac{n^{a_{0}+L}}{g^{m+L}}L!\ll\frac{n^{a_{0}}}{g^{m}}\cdot g^{-\frac{L}{4}}.\label{eq:bound large r}
\end{equation}
We now bound the number of $\left\{ \left(g_{1},a_{1},n_{1}\right),\dots,\left(g_{q},a_{q},n_{q}\right)\right\} \in\mathcal{A}$
with $\max_{1\leqslant i\leqslant q}\left(2g_{i}+a_{i}+n_{i}-3\right)$
$=2g+n-3-m-L$. Assume we have that $2g_{1}+a_{1}+n_{1}-3=2g+n-3-m-L$
. The remaining $q-1$ triples satisfy 
\[
\sum_{2\leqslant i\leqslant q}\left(2g_{i}+a_{i}+n_{i}\right)=L+3\left(q-1\right).
\]
Since $\sum_{i=1}^{q}n_{i}=k-2n_{0}$ and $\sum_{j=1}^{q}a_{j}=n-a_{0}$,
the triple $\left(g_{1},a_{1},n_{1}\right)$ is determined by the
choice of $\left\{ \left(g_{2},a_{2},n_{2}\right),\dots,\left(g_{q},a_{q},n_{q}\right)\right\} $.
Then the number of $\left\{ \left(g_{i},a_{i},n_{i}\right)\right\} _{i=1}^{q}\in\mathcal{A}$
with $\max_{1\leqslant i\leqslant q}\left(2g_{i}+a_{i}+n_{i}-3\right)=2g+n-3-m-L$
is therefore bounded above by 
\begin{equation}
{L+6\left(q-1\right) \choose 3\left(q-1\right)}.\label{eq:=000023combinations}
\end{equation}
Therefore combining (\ref{eq:Boundings}), (\ref{eq:bound large r})
and (\ref{eq:=000023combinations}) we see that the sum (\ref{eq:less than q})
satisfies
\begin{align}
 & \sum_{\substack{\left\{ \left(g_{i},a_{i},n_{i}\right)\right\} _{i=1}^{q}\in\mathcal{A}\\
\max_{1\leqslant i\leqslant q}\left(2g_{i}+a_{i}+n_{i}-3\right)>2g+n-3-m-8q
}
}\frac{C_{1}^{q}\sqrt{g}}{\prod_{j=1}^{q}\sqrt{g_{j}+\ensuremath{\overline{a_{j}+n_{j}}}}}\frac{n!\prod_{j=1}^{q}\left(2g_{j}+a_{j}+n_{j}-3\right)!}{\prod_{j=0}^{q}a_{j}!\left(2g+n-3\right)!}\nonumber \\
\ll & \frac{n^{a_{0}}}{g^{m}}\sum_{L=q-1}^{8q}{L+6\left(q-1\right) \choose 3\left(q-1\right)}\frac{C_{1}^{q}}{g^{\frac{L}{4}}}\ll\frac{n^{a_{0}}}{g^{m}}.\label{Summationpart2}
\end{align}
Combining (\ref{eq:Summationpart1}) and (\ref{Summationpart2}),
the result follows.
\end{proof}

\lyxaddress{Will Hide, \\
Department of Mathematical Sciences,\\
Durham University, \\
Lower Mountjoy, DH1 3LE Durham,\\
United Kingdom\\
\texttt{william.hide@durham.ac.uk}~\\
}
\end{document}